\documentclass[reqno]{article}
\usepackage{amsmath}
\usepackage{amsthm,amscd,amssymb,amsfonts, amsbsy}
\usepackage{latexsym}
\usepackage{mathrsfs}
\usepackage{pxfonts}
\usepackage{color}
\usepackage{enumerate}

\numberwithin{equation}{section}

\theoremstyle{plain}
\newtheorem{theorem}[equation]{Theorem}
\newtheorem{lemma}[equation]{Lemma}
\newtheorem{corollary}[equation]{Corollary}

\theoremstyle{definition}

\theoremstyle{remark}
\newtheorem{remark}[equation]{Remark}

\newcommand{\supp}{\operatorname{supp}}
\newcommand{\dist}{\operatorname{dist}}
\newcommand{\diam}{\operatorname{diam}}

\newcommand{\esssup}{\operatorname*{ess\,sup}}

\newcommand{\bR}{\mathbb R}

\newcommand\cD{\mathcal{D}}
\newcommand\cN{\mathcal{N}}

\newcommand\cQ{\mathcal{Q}}
\newcommand\cR{\mathcal{R}}

\newcommand\sB{\mathscr{B}}
\newcommand\sC{\mathscr{C}}

\newcommand\sL{\mathscr{L}}

\newcommand\sP{\mathscr{P}}

\newcommand\sV{\mathscr{V}}
\newcommand\sW{\mathscr{W}}

\providecommand{\ip}[1]{\langle#1\rangle}

\providecommand{\set}[1]{\{#1\}}
\providecommand{\Set}[1]{\left\{#1\right\}}
\providecommand{\bigset}[1]{\bigl\{#1\bigr\}}

\providecommand{\abs}[1]{\lvert#1\rvert}
\providecommand{\Abs}[1]{\left\lvert#1\right\rvert}
\providecommand{\bigabs}[1]{\bigl\lvert#1\bigr\rvert}

\providecommand{\norm}[1]{\lVert#1\rVert}

\providecommand{\tri}[1]{\lvert\!\lvert\!\lvert#1\rvert\!\rvert\!\rvert}

\renewcommand{\vec}[1]{\boldsymbol{#1}}
\renewcommand{\qedsymbol}{$\blacksquare$}

\newcommand{\traction}{\boldsymbol{\tau}}
\newcommand{\strain}{\varepsilon}

\begin{document}
\title{Heat kernel for the elliptic system of linear elasticity with boundary conditions}

\author{Justin Taylor\\Department of Mathematics\\Murray State University\\Murray, Kentucky  42071, USA
\and
Seick Kim\footnote{Seick Kim is supported by NRF Grant No. 2012-040411 and R31-10049 (WCU program).}\\Department of Computational Science and Engineering\\Yonsei University\\Seoul, 120-749, Korea
\and
Russell Brown\footnote{This work was partially supported by a grant from the Simons Foundation (\#195075 to Russell Brown).}\\Department of Mathematics\\University of Kentucky\\Lexington, Kentucky 40506, USA}
\maketitle

\abstract{We consider the elliptic system of linear elasticity with bounded measurable coefficients in a  domain where the second Korn inequality holds.
We construct heat kernel of the system subject to Dirichlet, Neumann, or mixed boundary condition under the assumption that weak solutions of the elliptic system are H\"older continuous in the interior.
Moreover, we show that if weak solutions of the mixed problem are H\"older continuous up to the boundary, then the corresponding heat kernel has a Gaussian bound.
In particular, if the domain is a two dimensional Lipschitz domain satisfying a corkscrew or non-tangential accessibility condition on the set where we specify Dirichlet boundary condition, then we show that the heat kernel has a Gaussian bound.
As an application, we construct Green's function for elliptic mixed problem in such a domain. 
}

\section{Introduction}
In a domain (i.e. a connected  open  set) $\Omega\subset \bR^n$ ($n \ge 2$), we consider the differential operator 
\begin{equation}		\label{eq1.01a}
L\vec u= \frac{\partial}{\partial x_\alpha} \left(A^{\alpha\beta}(x) \frac{\partial \vec u}{\partial x_\beta}\right),
\end{equation}
where $\vec u$ is a column vector with components $u^1,\ldots, u^n$, $A^{\alpha\beta}(x)$ are $n\times n$ matrices whose elements $a^{\alpha\beta}_{ij}(x)$ are bounded measurable functions which satisfy the symmetry condition
\begin{gather}
\label{eq1.02a}
a^{\alpha\beta}_{ij}(x)=a^{\beta\alpha}_{ji}(x)=a^{i \beta}_{\alpha j}(x),\\
\label{eq1.03a}
\frac{1}{4}\kappa_1 \bigabs{\vec \xi+\vec \xi^T}^2 \le a^{\alpha\beta}_{ij}(x) \xi^j_\beta \xi^i_\alpha \le \frac{1}{4}\kappa_2 \bigabs{\vec \xi+\vec \xi^T}^2,
\end{gather}
where $\vec \xi$ is an arbitrary $n\times n$ matrix with real entries $\xi^\alpha_i$, $\kappa_1, \kappa_2>0$, and $\abs{\vec \xi}=(\vec \xi,\vec \xi)^{1/2}$. 
Here and below, we will follow the convention that we sum  over repeated indices and notation $(\vec \xi,\vec \eta)=\xi_\alpha^i \eta_\alpha^{i}$
for $n\times n$ matrices $\vec \xi=(\xi_\alpha^i)$ and $\vec \eta=(\eta_\alpha^i)$.
The operator $L$ defined by \eqref{eq1.01a} can also be written in coordinate form as follows.
\[
(L \vec u)^i = \sum_{\alpha, \beta=1}^n \sum_{j=1}^n \frac{\partial}{\partial x_\alpha} \left(a^{\alpha\beta}_{ij}(x) \frac{\partial u^j}{\partial x_\beta}\right),\quad i=1,\ldots, n.
\]
Our assumptions on the coefficients $ a ^ { \alpha \beta } _ { ij}$ will include the equations of linear elasticity and in this case, 
the coefficients $a^{\alpha\beta}_{ij}(x)$ are usually referred to as the elasticity tensor; see e.g. \cite{Ciarlet, MH}.
In the classical theory of linear elasticity, the elasticity tensor for a homogeneous isotropic body is given by the formula
\[
a_{ij}^{\alpha\beta} = \lambda \delta_{i \alpha} \delta _{j \beta}+ \mu (\delta_{ij} \delta_{\alpha\beta} + \delta_{i \beta} \delta_{j \alpha}),
\]
where $\lambda>0$, $\mu>0$ are the Lam\'e constants, $\delta_{ij}$ is the Kronecker symbol.
In this case, the conditions \eqref{eq1.02a} and \eqref{eq1.03a} are satisfied with $\kappa_1=2\mu$ and $\kappa_2= 2\mu+n\lambda$.
We define the traction $\traction=\traction(\vec u)$ by the formula
\begin{equation}			\label{eq:traction}
\traction(\vec u)=\nu_\alpha  A^{\alpha\beta}(x) \frac{\partial \vec u}{\partial x_\beta},
\end{equation}
where $\vec \nu=(\nu_1,\ldots,\nu_n)$ is the unit outward normal to $\partial\Omega$.

We consider following boundary value problems, which are the ones most frequently considered in the theory of linear elasticity.

\begin{enumerate}[1.]
\item
Dirichlet (displacement) problem
\begin{equation}		\label{DP}\tag{DP}
L\vec u = \vec f\;\text{ in }\;\Omega,\quad \vec u =\vec \Phi\;\text{ on }\;\partial\Omega.
\end{equation}
\item
Neumann (traction) problem
\begin{equation}		\label{NP}\tag{NP}
L\vec u = \vec f\;\text{ in }\;\Omega,\quad \traction(\vec u)=\vec \varphi \;\text{ on }\;\partial\Omega.
\end{equation}
\item
Mixed problem
\begin{equation}		\label{MP}\tag{MP}
L\vec u = \vec f\;\text{ in }\;\Omega,\quad \vec u =\vec \Phi\;\text{ on }\;D,\quad \traction(\vec u) =\vec \varphi\;\text{ on }\;N,
\end{equation}
where $D$ and $ N$ are subsets of $\partial\Omega$ such that $D \cup N=\partial \Omega$ and $ D\cap N= \emptyset$.
\end{enumerate}
In the above, the equation as well as the boundary condition should be interpreted in  a weak sense; see Section~\ref{sec2} for a precise formulation.

In this article, we are concerned with the heat kernel  associated with the mixed problem \eqref{MP}.
By allowing $D=\partial\Omega, N=\emptyset$ and $D=\emptyset, N=\partial\Omega$ in \eqref{MP}, we may regard \eqref{DP} and \eqref{NP} as extreme cases of \eqref{MP} and this is why we focus on \eqref{MP}.
By the heat kernel for \eqref{MP}, we mean an $n\times n$ matrix valued function $\vec K(x,y,t)$ satisfying
\begin{equation}			\label{HK}
\left\{
\begin{array}{cl}
\frac{\partial}{\partial t} \vec K(x,y,t)-  L_x \vec K(x,y,t) = 0 \quad & \text{in }\; \Omega \times (0,\infty),\\
\vec K(x,y,t)=0 & \text{on }\;  D \times (0,\infty),\\
\traction_x(\vec K(x,y,t))=0 & \text{on }\;  N \times (0,\infty),\\
\vec K(x,y,0) = \delta_y \vec I &  \text{on }\; \Omega,
\end{array}
\right. 
\end{equation}
where $\delta_{y}(\cdot)$ is Dirac delta function concentrated at $y$,  $\vec I$ is the $n\times n$ identity matrix, and the equation as well as the boundary condition should be interpreted in some weak sense.
The heat kernels for \eqref{DP} and \eqref{NP} are similarly defined and they are frequently referred to as Dirichlet and Neumann heat kernels.

We assume that $\Omega$ is a bounded $(\epsilon,\delta)$-domain of Jones \cite{Jones} and if $D\neq \emptyset$ and $D\neq \partial\Omega$, we assume further that it has a Lipschitz portion; if $D=\partial\Omega$, then we require none of these conditions (see H1 in Section~\ref{sec:bac}). 
We prove that if weak solutions of the system $L\vec u =0$ are locally H\"{o}lder continuous (see H2 in Section~\ref{sec:bac}), then the heat kernel for \eqref{MP} exists and satisfies a natural growth estimate near the pole; see Theorem~\ref{thm1}.
It is known that weak solutions of $L$ are H\"{o}lder continuous if $n=2$ or if the coefficients are uniformly continuous.
We also prove that if the gradient of weak solutions of the system $L \vec u=0$ satisfy the growth condition called \emph{Dirichlet property} (see H3 in Section~\ref{sec:bac}), then the heat kernel has a Gaussian upper bound; see Theorem~\ref{thm3}.  The 
Dirichlet property is known to hold in the case when $\Omega$ is a Lipschitz domain in $\bR^2$ and $D$ satisfies the \emph{corkscrew condition} (see \cite{TKB}) or when the coefficients and domains are sufficiently smooth and $\bar D\cap \bar N=\emptyset$ (see \cite{Fichera}).
As an application, we construct Green's function for the elliptic system from the heat kernel and in the presence of Dirichlet property, we show that the Green's function has the usual bound of $C\abs{x-y}^{2-n}$ (or logarithmic bound if $n=2$); see Theorem~\ref{thm4}.

A few remarks are in order.
The Dirichlet or Neumann heat kernels for elliptic equations are studied by many authors; see Davies \cite{Davies89}, Robinson \cite{Robinson}, Varopoulos et al. \cite{VSC}, and references therein. 
For the heat kernel for second-order elliptic operators in divergence form satisfying Robin-type boundary conditions, we mention Gesztesy et al. \cite{GMN}.
 Dirichlet and Neumann Green's functions for strongly  parabolic systems are studied in Cho et al. \cite{CDK, CDK2} and Choi and Kim \cite{CK2}.
We should mention that  our paper, though   technically more
involved, is an extension of  their method.
The elliptic  Green   function for \eqref{MP} in two dimensional domains is constructed in Taylor et al. \cite{TKB} by a different method not involving the heat kernel.
Previously, Taylor et al. \cite{TOB} give a construction of the Green function for a class of mixed problems for the Laplacian in a Lipschitz domain in dimensions two and higher.

In recent years, there has been increasing interest in the study of
elliptic equations under mixed boundary conditions from a variety of
viewpoints. 
Haller-Dintelmann et al. 
\cite{MR2551487}
consider H\"older continuity of solutions of a single equation with an
interest in applications in control. 
Mazzucato and Nistor \cite{MR2564468} study the mixed problem for the
elliptic system of elasticity in
polyhedral domains under mixed boundary conditions. 
Finally, a recent monograph of Maz'ya and Rossmann \cite{MR2641539}
treats mixed problems for systems in polyhedral domains. We refer to
the above works and their references for background on the study of
mixed problems for systems. 
It would be interesting to see if the well-posedness results in these
works can be used to obtain
fundamental estimates H2 and H3 that we use to obtain further
estimates for the heat kernel.

The organization of the paper is as follows. In Section~\ref{sec2}, we introduce some notation and definitions including the precise definition of the heat kernel for \eqref{MP}.
In Section~\ref{sec:main}, we state our main theorems (Theorems~\ref{thm1} and \ref{thm3}), which we briefly described above.
In Section~\ref{sec:apps}, we construct, as an application, the Green's function for \eqref{MP} and obtain the usual bounds.
We give the proofs for our main results in Section~\ref{sec:pf} and some technical lemmas are proved in Appendix.

\section{Preliminaries}	\label{sec2}
\subsection{Notation and definition}
Throughout the article, we let $\Omega$ denote a domain in $\bR^n$ and let $D$ and $N$ be fixed subsets of $\partial\Omega$ such that $D \cup N =\partial\Omega$ and $D \cap N = \emptyset$.
We use $X=(x,t)$ to denote a point in $\bR^{n+1}$; $x=(x_1,\ldots, x_n)$ will always be a point in $\bR^n$.
We also write $Y=(y,s)$, $X_0=(x_0,t_0)$, and reserve notation
\[
\hat Y=(y,0)=(y_1,\ldots,y_n,0).
\]
We define the parabolic distance in $\bR^{n+1}$  by
\[
\abs{X-Y}_{\sP}= \max(\abs{x-y}, \sqrt{\abs{t-s}}),
\]
where $\abs{\,\cdot\,}$ denotes the usual Euclidean norm, and write $\abs{X}_{\sP}=\abs{X-0}_{\sP}$.
We use the following notation for basic cylinders in $\bR^{n+1}$.
\begin{gather*}
Q(X,r)=\set{Y\in \bR^{n+1}\colon  \abs{Y-X}_{\sP}<r},\\
Q_{-}(X,r)=\set{Y=(y,s)\in \bR^{n+1}\colon \abs{Y-X}_{\sP}<r,\, s<t},\\
Q_{+}(X,r)=\set{Y=(y,s)\in \bR^{n+1}\colon \abs{Y-X}_{\sP}<r,\, s>t}.
\end{gather*}
We also use $B(x,r)=\set{y\in \bR^n\colon \abs{y-x}<r}$ to denote a
ball in $ \bR^n$.  For the vector valued function $\vec u=(u^1,\ldots,
u^n)^T$, we denote by $\strain(\vec u)$ (called the strain tensor) the matrix whose elements
are
\[
\strain_{ij}(\vec u)=\frac{1}{2}\left(\frac{\partial u^i}{\partial x_j}+\frac{\partial u^j}{\partial x_i}\right).
\]
We denote by $\cR$ the linear space of rigid displacements of $\bR^n$; i.e.
\begin{equation}		\label{spacer}
\cR=\set{\vec A x+\vec b \colon \text{$\vec A$ is real $n\times n$ skew-symmetric matrix and $\vec b \in \bR^n$}}
\end{equation}
Note that $\cR$ is a real vector space of dimension $N=n(n+1)/2$ and that
\begin{equation}		\label{l2ip}
\ip{\vec u, \vec v} = \int_\Omega \vec u\cdot \vec v \,dx=\int_\Omega \vec v^T \vec u \,dx
\end{equation}
defines an inner product on $\cR$.
Let us fix an orthonormal basis $\set{\vec \omega_{i}}_{i=1}^N$ in $\cR$ and define the projection operator $\pi_\cR \colon L^2(\Omega)^n\to \cR$ by
\begin{equation}			\label{pr}
\pi_\cR(\vec u)= \sum_{i=1}^N \ip{\vec u, \vec \omega_i} \vec \omega_i.
\end{equation}
The above formula still makes sense for $\vec u= \delta_y \vec e_k$, where $\vec e_k$ is the $k$th unit (column) vector in $\bR^n$.
For $y\in \Omega$, we denote by $\vec T_y=\vec T_y(x)$ an $n\times n$ matrix valued function such that
\begin{equation}			\label{eq2.04sj}
\vec T_y \vec e_k=
\begin{cases}  \pi_\cR(\delta_y \vec e_k) & \text{if $D=\emptyset$,}
\\
0 &\text{if $D \neq \emptyset$.}
\end{cases}
\end{equation}
Roughly speaking, $\vec T_y$ is an orthogonal projection of $\delta_y \vec I$ on $\cR$ if $D=\emptyset$ and $0$ otherwise.
We set
\[
\sB(\vec u,\vec v)= a^{\alpha\beta}_{ij} \frac{\partial u^j}{\partial x_\beta} \frac{\partial v^i}{\partial x_\alpha}.
\]
It follows from \eqref{eq1.02a} the form  $\sB$ is symmetric (i.e., $\sB(\vec u,\vec v)=\sB(\vec v,\vec u)$) and from \eqref{eq1.03a} that
\begin{equation}			\label{eq1.10b}
\kappa_2^{-1} \sB(\vec u,\vec u)\le \abs{\strain(\vec u)}^2 \le \kappa_1^{-1} \sB(\vec u,\vec u).
\end{equation}
It is easy to verify that \eqref{eq1.02a}, \eqref{eq1.03a} imply (see \cite[Lemma~3.1, p. 30]{OSY})
\[
a^{\alpha\beta}_{ij} \xi^j_\beta \eta^i_\alpha \le \frac{1}{4}\kappa_2 \abs{\vec \xi+\vec \xi^T}\, \abs{\vec \eta+\vec \eta^T} \le \kappa_2 \abs{\vec \xi}\abs{\vec \eta}.
\]
Finally, we denote $d_x=d(x)=\dist(x, \Omega^c)$ when $\Omega$ is clear from the context and  write $a\wedge b=\min(a,b)$ and $a \vee b=\max(a,b)$ for $a, b \in 
\bar \bR$.

\subsection{Function spaces}
We use the notation in Gilbarg \& Trudinger \cite{GT} for the standard functions spaces defined on $U\subset \bR^n$ such as $L^p(U)$, $W^{k,p}(U)$, $C^{k,\alpha}(\bar U)$, etc.
For $\Gamma\subset \partial U$, let $W^{1,2}(U; \Gamma)$ be the subspace obtained by taking the closure in $W^{1,2}(U)$ of smooth functions in $\bar U$ which vanish in a neighborhood of $\Gamma$.
Note that we have $W^{1,2}(U;\partial U)=W^{1,2}_0(U)$.
We shall denote
\[
\tilde{W}^{1,2}(U;\Gamma):= 
\begin{cases}
W^{1,2}(U;\Gamma) &\mbox{if } \Gamma \neq \emptyset \\
\Set{u\in W^{1,2}(U) \colon \int_U u\,dx=0 } & \mbox{if }  \Gamma=\emptyset. 
\end{cases} 
\]
For $\Omega$ and $D$ as above, we define
\begin{equation}		\label{eqv}
\vec V :=
\begin{cases}
W^{1,2}(\Omega; D)^n &\mbox{if } D \neq \emptyset \\
\bigset{\vec u\in W^{1,2}(\Omega)^n \colon \ip{\vec u, \vec v} =0,\;\; \forall \vec v \in \cR } & \mbox{if }  D=\emptyset,
\end{cases} 
\end{equation}
where $\cR$ and $\ip{\vec u,\vec v}$ are as in \eqref{spacer} and \eqref{l2ip}.
Notice that $\vec V \subset \tilde{W}^{1,2}(\Omega; D)^n$.
For $\mu \in (0,1]$, we denote
\[
\abs{u}_{\mu;U}=[u]_{\mu;U}+\abs{u}_{0;U}=\sup_{\substack{x, y \in U\\ x\neq y}}\frac{\abs{u(x)-u(y)}}{\abs{x-y}^\mu}+\sup_{x\in U}\,\abs{u(x)}.
\]

For spaces of functions defined on $Q\subset \bR^{n+1}$, we borrow notation mainly from Ladyzhenskaya et al. \cite{LSU}.
To avoid confusion, spaces of functions defined on $Q\subset \bR^{n+1}$ shall be always written in \emph{script letters}.
For $q\ge 1$, we let $\sL_q(Q)$ denote the Banach space consisting of measurable functions on $Q$ that are $q$-integrable.
For $Q=\Omega\times (a,b)$, we denote by $\sL_{q,r}(Q)$ the Banach space consisting of all measurable functions on $Q$ with a finite norm
\[
\norm{u}_{\sL_{q,r}(Q)}=\left(\int_a^b\left(\int_{\Omega} \abs{u(x,t)}^q\, dx\right)^{r/q}dt\right)^{1/r},
\]
where $q\ge 1$ and $r\ge 1$. Thus
$\sL_{q,q}(Q)$ is the space  $\sL_q(Q)$.
By $\sC^{\mu,\mu/2}(\bar Q)$ we denote the set of all bounded measurable functions $u$ on $Q$ for which $\abs{u}_{\mu,\mu/2;Q}$ is finite, where
we define the parabolic H\"{o}lder norm as follows:
\begin{align*}
\abs{u}_{\mu,\mu/2;Q}&=[u]_{\mu,\mu/2;Q}+\abs{u}_{0;Q}\\
&=\sup_{\substack{X, Y \in Q\\ X\neq Y}}\frac{\abs{u(X)-u(Y)}}{\abs{X-Y}^\mu_\sP}+\sup_{X\in Q}\abs{u(X)}, \quad \mu\in(0,1].
\end{align*}
We write $u\in \sC^\infty_c(Q)$ (resp. $\sC^\infty_c(\bar Q)$) if $u$ is an infinitely differentiable function on $\bR^{n+1}$ with  compact support in $Q$ (resp. $\bar Q$).
We write $D_i u=\partial u/\partial x_i$ ($i=1,\ldots,n$) and $u_t=\partial u /\partial t$.
We also write $Du=D_x u=(D_1 u,\ldots, D_n u)$.
We write $Q(t_0)$ for the set of all points $(x,t_0)$ in $Q$ and $I(Q)$ for
the set of all $t$ such that $Q(t)$ is nonempty.
We denote
\[
\tri{u}_{Q}^2= \int_{Q}  \abs{D_x u}^2 \,dx \, dt+\esssup\limits_{t\in I(Q)} \int_{Q(t)} \abs{u(x,t)}^2\,dx.
\]
The space $\sW^{1,0}_q(Q)$ denotes the Banach space consisting of functions $u\in \sL_q(Q)$ with weak derivatives $D_i u \in \sL_q(Q)$ ($i=1,\ldots,n$) with the norm
\[
\norm{u}_{\sW^{1,0}_q(Q)}=\norm{u}_{\sL_q(Q)}+\norm{D_x u}_{\sL_q(Q)}
\]
and by $\sW^{1,1}_q(Q)$ the Banach space with the norm
\[
\norm{u}_{\sW^{1,1}_q(Q)}=\norm{u}_{\sL_q(Q)}+\norm{D_x u}_{\sL_q(Q)}+ \norm{u_t}_{\sL_q(Q)}.
\]
In the case when $Q$ has a finite height (i.e., $Q\subset \bR^{n}\times(-T,T)$ for some $T<\infty$), we define $\sV_2(Q)$ as the Banach space consisting of all elements of $\sW^{1,0}_2(Q)$ having a finite norm $\norm{u}_{\sV_2(Q)}:=  \tri{u}_{Q}$ and the space $\sV^{1,0}_2(Q)$ is obtained by completing the set $\sW^{1,1}_2(Q)$ in the norm of $\sV_2(Q)$.
When $Q$ does not have a finite height, we say that $u \in \sV_2(Q)$ (resp. $\sV^{1,0}_2(Q)$) if $u \in \sV_2(Q_T)$ (resp. $\sV^{1,0}_2(Q_T)$) for all $T>0$, where $Q_T=Q \cap \set{ \abs{t} <T}$, and $\tri{u}_Q <\infty$.
Note that this definition allows that $1\in \sV^{1,0}_2(\Omega\times(0,\infty))$.
Finally, we write $u\in \sL_{q,loc}(Q)$ if $u \in \sL_q(Q')$ for all $Q'\Subset Q$ and similarly define $\sW^{1,0}_{q,loc}(Q)$, etc.

\subsection{Weak solutions}			\label{sec:weaksol}

For $\vec f,  \vec g_\alpha \in L^2(U)^n$, where $\alpha=1,\ldots, n$, we say that $\vec u$ is a weak  solution of $L \vec u=\vec f+D_\alpha \vec g_\alpha$ in $U$ if $\vec u \in W^{1,2}(U)^n$ and for any $\vec v \in W^{1,2}_0(U)^n$ satisfies the identity
\begin{equation}			\label{eq2.02bt}
\int_{U}  \sB(\vec u, \vec v)=-\int_{U} \vec f \cdot \vec v+\int_{U} \vec g_\alpha \cdot D_\alpha\vec v.
\end{equation}
Let $\Gamma_D, \Gamma_N$ be disjoint subsets of $\partial U$.
Recall that the traction $\traction(\vec u)$ is defined by the formula \eqref{eq:traction}.
We say that $\vec u$ is a weak solution of
\[
L\vec u = \vec f+D_\alpha \vec g_\alpha \;\text{ in }\;U,\quad \vec u =0\;\text{ on }\;\Gamma_D,\quad \traction(\vec u) =\vec g_\alpha \nu_\alpha\;\text{ on }\;\Gamma_N
\]
if $\vec u \in W^{1,2}(U;\Gamma_D)^n$ and for any $\vec v \in W^{1,2}(U;\partial  U\setminus \Gamma_N)^n$ satisfies the identity \eqref{eq2.02bt}.
Let $\Omega, D, N$ be as above.
For $\vec f\in L^2(\Omega)^n$, we say that $\vec u$ is a weak solution of the mixed problem
\[
L\vec u = \vec f\;\text{ in }\;\Omega,\quad \vec u =0\;\text{ on }\;D,\quad \traction(\vec u) =0\;\text{ on }\;N
\]
if $\vec u \in W^{1,2}(\Omega;D)^n$ and satisfies for any $\vec v \in W^{1,2}(\Omega; D)^n$ the identity
\[
\int_\Omega \sB(\vec u, \vec v)= -\int_\Omega \vec f\cdot \vec v.
\]

Let $Q$ be a cylinder $U \times (a,b)$, where $U \subset \bR^n$ and $-\infty< a<b <\infty$.
For $\vec f \in \sL_{2,1}(Q)^n$ and $\vec g_\alpha \in \sL_2(Q)^n$, we say that $\vec u$ is a weak solution of 
\[
\vec u_t-L\vec u=\vec f + D_\alpha \vec g_\alpha
\]
if $\vec u \in \sV_2(Q)^n$ and satisfies for all $\vec \phi \in \sC^\infty_c (Q)^n$ the identity
\begin{equation}		\label{eqn:E-71}
-\int_{Q} \vec u \cdot \vec \phi_t+ \int_{Q} \sB(\vec u, \vec \phi)= \int_{Q} \vec f \cdot \vec \phi-\int_{Q} \vec g_\alpha \cdot D_\alpha\vec \phi.
\end{equation}
Let $\Gamma_D, \Gamma_N$ be disjoint subsets of $\partial U$.
We say that $\vec u$ is a weak solution of
\[
\left\{
\begin{array}{ll}
\vec u_t - L\vec u =\vec f + D_\alpha \vec g_\alpha & \text{in }\; U  \times (a,b) =:Q\\
\vec u=0 & \text{on }\;  \Gamma_D \times (a,b)\\
\traction(\vec u) = -\vec g_\alpha \nu_\alpha  & \text{on }\; \Gamma_N \times (a,b) =:S
\end{array}
\right. 
\]
if $\vec u\in \sV_2(Q)^n$, $\vec u(\cdot, t) \in W^{1,2}(U; \Gamma_D)^n$ for a.e. $t\in (a,b)$, and it satisfies the identity \eqref{eqn:E-71} for all $\vec \phi \in \sC^\infty_c (Q\cup S)^n$.
Next, denote $Q=\Omega\times (a,b)$, and let $\vec f \in \sL_{2,1}(Q)^n$, $\vec g_\alpha\in \sL_2(Q)^n$, and $\vec \psi \in L^2(\Omega)^n$.
By a weak solution in $\sV_2(Q)^n$ (resp. $\sV^{1,0}_2(Q)^n$) of the problem
\begin{equation} \label{PMixed}
\left\{
\begin{array}{ll}
\vec u_t - L\vec u =\vec f + D_\alpha \vec g_\alpha& \text{in }\; \Omega \times (a,b)\\
\vec u=0 & \text{on }\;  D\times (a,b)\\
\traction(\vec u) = -\vec g_\alpha \nu_\alpha & \text{on }\; N \times (a,b)\\
\vec u(\cdot, a) = \vec \psi &  \text{on }\;\Omega,
\end{array}
\right. 
\end{equation}
we mean $\vec u(x,t)$ in $\sV_2(Q)^n$ (resp. $\sV^{1,0}_2(Q)^n$) such that  $\vec u(\cdot, t) \in W^{1,2}(\Omega;D)^n$ for a.e. $t\in (a,b)$ and satisfying the identity
\[
-\int_Q  \vec u \cdot \vec \phi_t +\int_Q \sB(\vec u,\vec \phi) - \int_Q \vec f \cdot \vec \phi+\int_{Q} \vec g_\alpha \cdot D_\alpha\vec \phi = \int_{\Omega} \vec \psi(x) \cdot \vec \phi(x,a)\,dx
\]
for any $\vec \phi(x,t) \in \sC^\infty_c(\bar Q)^n$ that vanishes on $D\times (a,b)$ and  equals to zero for $t=a$.
We may identify $\vec f\in \sL_{2,1}(Q)^n$ as an element in $L^1(a,b; \vec V')$ in the sense
\[
[\vec f(\cdot, t)] (\vec v) =\int_\Omega \vec f(x,t)\cdot \vec v(x)\,dx,\quad \vec v \in \vec V,
\]
and consider the problem
\begin{equation} \label{eq2.12mx}
\left\{
\begin{array}{ll}
\vec u_t - L\vec u =\vec f & \text{in }\; \Omega \times (a,b)\\
\vec u=0 & \text{on }\;  D\times (a,b)\\
\traction(\vec u) =0 & \text{on }\; N \times (a,b)\\
\vec u(\cdot, a) = \vec \psi &  \text{on }\;\Omega\\
\vec u(\cdot, t)\in \vec V&\text{for a.e.}\;t\in (a,b).
\end{array}
\right. 
\end{equation}
In the above, we impose the compatibility condition for $\vec \psi$ that $\ip{\vec \psi,\vec v}=0$ for all $\vec v \in \cR$ in the case when $D=\emptyset$.
We shall say that $\vec u$ is a weak solution of the problem \eqref{eq2.12mx} if $\vec u \in\sV_2^{1,0}(Q)^n$, $\vec u(\cdot, t) \in \vec V$ for a.e. $t\in(a,b)$, and satisfies the identity
\[
-\int_Q  \vec u \cdot \vec \phi_t +\int_Q \sB(\vec u,\vec \phi) - \int_Q \vec f \cdot \vec \phi  = \int_{\Omega} \vec \psi(x) \cdot \vec \phi(x,a)\,dx
\]
for any $\vec \phi(x,t) \in \sC^\infty_c(\bar Q)^n$ that vanishes on $D\times (a,b)$, satisfies $\ip{\vec \phi(\cdot, t), \vec v}= 0$ for all $t\in [a,b]$ and any $\vec v \in \cR$, and equals to zero for $t=a$.
Note that if $D\neq \emptyset$, then a weak solution of the problem \eqref{eq2.12mx} is also a weak solution in $\sV^{1,0}_2(Q)^n$ of the problem \eqref{PMixed} with $\vec g_\alpha=0$ for all $\alpha=1,\ldots,n$,  and vice versa.
However, when $D=\emptyset$, they are not the same in general; note that if $\vec f(\cdot,t) -\tilde{\vec f}(\cdot,t) \in \cR$, then they are the same as elements in $L^1(a,b; \vec V')$.

\subsection{Heat kernel for the system of linear elasticity}	\label{sec2.4}
We say that an $n\times n$ matrix valued function $\vec K(x,y,t)$, with measurable entries $K_{ij} \colon \Omega\times \Omega\times [0,\infty) \to  \bar \bR$, is the heat kernel for \eqref{MP} if it satisfies the following properties, where we denote  $\cQ=\Omega\times [0,\infty)$.
\begin{enumerate}[a)]
\item
For all $y\in \Omega$, elements of $\vec K(\cdot,y,\cdot)$ belong to $\sW^{1,0}_{1,loc}(\cQ) \cap \sV_{2}(\cQ \setminus Q_{+}(\hat{Y},r))$ for any $r>0$.

\item
For all $y\in \Omega$, $\vec K(\cdot,y,\cdot)$ is a generalized solution of the problem \eqref{HK} in the sense that $\vec K(\cdot,y,t) \in W^{1,2}(\Omega;D)^n$ for a.e. $t>0$ and for any $\vec \phi=(\phi^1,\ldots,\phi^n)^T \in \sC^\infty_c(\bar \cQ)^n$ that vanishes on $N\times [0,\infty)$, we have the identity
\begin{multline}		\label{eq2.04x}
-\int_{\cQ} K_{ik}(x,y,t) \frac{\partial}{\partial t} \phi^i (x,t)\,dx\,dt\\+ \int_{\cQ} a^{\alpha\beta}_{ij}\frac{\partial}{\partial x_\beta} K_{jk}(x,y,t) \frac{\partial}{\partial x_\alpha} \phi^i(x,t) \,dx\,dt
= \phi^k(y,0).
\end{multline}
\item
For any $\vec f=(f^1,\ldots, f^n)^T \in \sC^\infty_c(\bar \cQ)^n$, the function $\vec u$ given by
\[
\vec u(x,t):= \int_0^t\!\!\!\int_\Omega \vec K(y,x,t-s)^T \vec f(y,s)\,dy\,ds
\]
is, for any $T>0$, a unique weak solution in $\sV^{1,0}_2(\Omega\times (0,T))^n$ of the problem 
\begin{equation}			\label{2.13vv}
\left\{
\begin{array}{ll}
\vec u_t - L \vec u = \vec f \quad & \text{in }\; \Omega \times (0,T), \\
\vec u=0 & \text{on }\;  D\times (0,T),\\
\traction(\vec u) =0  & \text{on }\; N \times (0,T),\\
\vec u(\cdot, 0) = 0  &  \text{on }\; \Omega.
\end{array} \right. 
\end{equation}
\end{enumerate}
We note that part c) of the above definition gives the uniqueness of the heat kernel for \eqref{MP}.

\subsection{Basic Assumptions and their consequences}	\label{sec:bac}
\begin{enumerate}[{\bf{H}1.}]
\item	
We assume $\Omega\subset \bR^n$, $n\ge 2$ is a bounded domain.
If $D=\partial\Omega$, we do not make any further assumption.
Otherwise, we assume the $(\epsilon,\delta)$-condition of Jones \cite{Jones} for some $\epsilon, \delta >0$:
For any $x, y \in \Omega$ such that $\abs{x-y}<\delta$, there is a rectifiable arc $\gamma$ joining $x$ to $y$ and satisfying
\[
l(\gamma) \le \frac{1}{\epsilon}\abs{x-y},\quad d(z) \ge \frac{\epsilon\abs{x-z}\abs{y-x}}{\abs{x-y}}\;\text{ for all $z$ on $\gamma$},
\]
where $l(\gamma)$ denotes the arc length of $\gamma$ and $d(z)$ is the distance from z to the complement of $\Omega$.
If $D \neq \partial\Omega$ and $D\neq \emptyset$, we assume further that $D$ has a portion of Lipschitz boundary; i.e. there exist $x_0\in D$ and a neighborhood $V$ of $x_0$ in $\bR^n$ and new orthogonal coordinates $\set{y_1,\ldots, y_n}$ such that $V$ is a hypercube in the new coordinates:
\[
V=\set{(y_1,\ldots, y_n): -a_j<y_j<a_j, \;1\le j \le n};
\]
there exists a Lipschitz continuous function $\varphi$ defined in
\[
V'=\set{(y_1,\ldots, y_{n-1}): -a_j<y_j<a_j, \;1\le j \le n-1};
\]
and such that
\[
\begin{aligned}
\abs{\varphi(y')} &\le a_n/2\;\text{ for every }\;y'=(y_1, \ldots, y_{n-1}) \in V',\\
\Omega\cap V &=\set{y=(y',y_n)\in V: y_n <\varphi(y')},\\
D \cap V &=\set{y=(y',y_n)\in V: y_n =\varphi(y')}.\\
\end{aligned}
\]
In other words, in a neighborhood $V$ of $x_0$, $\Omega$ is below the graph of $\varphi$ and $D$ is the graph of $\varphi$.
\end{enumerate}
Basically, we introduce the assumption H1 is to guarantee the multiplicative inequality \eqref{eq:mei} and the second Korn inequality \eqref{eq2.20ns} are available to us.
We recall that the following multiplicative inequality holds for any $u$  in $W^{1,2}(\bR^n)$ with $n\ge 1$; see \cite[Theorem~2.2, p. 62]{LSU}.
\begin{equation}		\label{extrn}
\norm{u}_{L^{2(n+2)/n}(\bR^n)} \le c(n) \norm{D u}_{L^2(\bR^n)}^{n/(n+2)}  \norm{u}_{L^2(\bR^n)}^{2/(n+2)}.
\end{equation}
If we assume H1, then there is an extension operator $E \colon W^{1,2}(\Omega) \to W^{1,2}(\bR^n)$ such that the following holds; see \cite[Theorem~8]{Rogers}.
\begin{equation}		\label{extop}
\norm{Eu}_{L^2(\bR^n)} \le C \norm{u}_{L^2(\Omega)},\quad
\norm{Eu}_{W^{1,2}(\bR^n)} \le C \norm{u}_{W^{1,2}(\Omega)}.
\end{equation}
Then by combining \eqref{extop} and \eqref{extrn}, for any $u \in \tilde W^{1,2}(\Omega; D)$, we obtain
\begin{align*}
\norm{u}_{L^{2(n+2)/n}(\Omega)} 
&\le C\norm{D(E u)}_{L^2(\bR^n)}^{n/(n+2)} \norm{E  u}_{L^2(\bR^n)}^{2/(n+2)}\\
& \le C\norm{u}_{W^{1,2}(\Omega)}^{n/(n+2)} \norm{u}_{L^2(\Omega)}^{2/(n+2)}\le C\norm{D u}^{n/(n+2)}_{L^2(\Omega)}\norm{u}^{2/(n+2)}_{L^2(\Omega)},
\end{align*}
where in the last step we used H1 to apply the Friedrichs inequality (or Poincar\'{e}'s inequality if $D=\emptyset$):
\[
\norm{u}_{L^2(\Omega)} \le C \norm{Du}_{L^2(\Omega)}, \quad \forall u \in \tilde{W}^{1,2}(\Omega; D).
\]
We have proved that H1 implies that there is $\gamma=\gamma(n,\Omega,D)$ such that for any $u \in \tilde W^{1,2}(\Omega; D)$, we have
\begin{equation}			\label{eq:mei}
\norm{u}_{L^{2(n+2)/n}(\Omega)} \le \gamma \norm{D u}^{n/(n+2)}_{L^2(\Omega)}\norm{u}^{2/(n+2)}_{L^2(\Omega)}.
\end{equation}
If $u \in \sV_2(\Omega\times (a,b))$ is such that $u(\cdot, t) \in \tilde W^{1,2}(\Omega;D)$ for a.e. $t \in (a,b)$, where $-\infty\le a <b\le \infty$, then by \eqref{eq:mei} we have (see \cite[pp. 74--75]{LSU})
\begin{equation}			\label{eq5.13d}
\norm{u}_{\sL_{2(n+2)/n}(\Omega\times(a,b))} \le \gamma  \tri{u}_{\Omega\times(a,b)}.
\end{equation}
Another important consequence of the inequality \eqref{eq:mei} is the following: For any $u\in \sV_2(\Omega\times(a,b))$ with $b-a<\infty$, we have
\begin{equation}			\label{paraemb}
\norm{u}_{\sL_{2(n+2)/n}(\Omega\times(a,b))} \le \left(2\gamma+(b-a)^{n/2}\abs{\Omega}^{-1}\right)^{\frac{1}{n+2}} \tri{u}_{\Omega\times(a,b)}.
\end{equation}
We refer to \cite[Eq.~(3.8), p. 77]{LSU} for the proof of \eqref{paraemb}. 
Moreover, H1 implies the following second Korn inequality: (see \cite{DM04} for the proof)
\begin{equation}		\label{eq2.20ns}
\norm{\vec u}_{W^{1,2}(\Omega)} \le C \left\{\norm{\vec u}_{L^2(\Omega)} +\norm{\strain(\vec u)}_{L^2(\Omega)} \right\}.
\end{equation}
In fact, if $\vec u \in W^{1,2}(\Omega;\partial\Omega)^n=W^{1,2}_0(\Omega)^n$, we have the first Korn inequality
\[
\norm{D \vec u}_{L^2(\Omega)}^2 \le 2 \norm{\strain(\vec u)}_{L^2(\Omega)}^2. 
\]
Also, we have the following inequalities for any $\vec u \in \vec V$:
\begin{equation}
\label{eq2.03a}
\norm{\vec u}_{W^{1,2}(\Omega)}  \le C \norm{\strain(\vec u)}_{L^2(\Omega)}.
\end{equation}
The inequality \eqref{eq2.03a} is obtained by utilizing \eqref{eq2.20ns} in the proof of \cite[Theorem~2.7, p. 21]{OSY}.
By \eqref{eq2.03a} and \eqref{eq1.10b}, for any  $\vec u\in \vec V$, we have
\begin{equation}			\label{eq2.08q}
\int_\Omega \sB(\vec u,\vec u)\,dx \ge c \int_\Omega \abs{D \vec u}^2 \,dx.
\end{equation}

\begin{lemma}		\label{lem:pre}
Assume  H1 and let $\vec \psi \in L^2(\Omega)^n$ and $\vec f \in \sL_{2,1}(Q)^n$, where $Q=\Omega\times (a,b)$ and $-\infty<a<b<\infty$.
Then, there exists a unique weak solution $\vec u$ in $\sV^{1,0}_2(Q)^n$ of the problem \eqref{PMixed}. 
Moreover, if we assume that $\ip{\vec \psi, \vec v}=0$ for all $\vec v \in \cR$ in the case  when $D=\emptyset$, then there also exists a unique weak solution of the problem \eqref{eq2.12mx}.
If $\norm{\vec f}_{\sL_{2(n+2)/(n+4)}(Q)}<\infty$, then the weak solution $\vec u$ of the problem \eqref{eq2.12mx} satisfies an energy inequality
\begin{equation}			\label{enieq}
\tri{\vec u}_{\Omega\times (a,b)} \le C \left\{ \norm{\vec f}_{\sL_{2(n+2)/(n+4)}(Q)}+ \norm{\vec\psi}_{L^2(\Omega)} \right\},
\end{equation}
where $C$ depends only on $n, \kappa_1, \kappa_2$ and the constants appearing in \eqref{eq5.13d} and \eqref{eq2.03a}.
\end{lemma}
\begin{proof}
With the aid of the second Korn inequalities \eqref{eq2.20ns} or
\eqref{eq2.03a}, it follows from the standard Galerkin's method and the energy inequality.
\end{proof}

\begin{enumerate}[{\bf{H}1.}]
\setcounter{enumi}{1}
\item
There exist $\mu_0\in(0,1]$ and $A_0>0$ such that if $\vec u$ is a weak solution of $L\vec u= 0$ in $B=B(x_0,r)$, where $x_0\in \Omega$ and $0<r\le d(x_0)$, then $\vec u$ is H\"{o}lder continuous in $\frac{1}{2}B=B(x_0,r/2)$ with an estimate
\begin{equation}		\label{IH}
[\vec u]_{\mu_0; \frac{1}{2}B}\le A_0 r^{-\mu_0}\left(\fint_{B}\,\abs{\vec u(y)}^2\,dy \right)^{1/2}.
\end{equation}
Here, we use the notation $\fint_B u \,dx =\frac{1}{\abs{B}}\int_B u\,dx$.
\end{enumerate}
It follows from H2 that a weak solution of $\vec u_t -L \vec u$ is
also locally  H\"{o}lder continuous.
 \begin{lemma}				\label{lem:ihp}
H2 implies that there exist $\mu_1\in (0,\mu_0)$ and $A_1>0$ such that whenever $\vec u$ is a weak solution in $\sV_2(Q)^n$ of $\vec u_t -L \vec u=0$ in $Q=Q_{-}(X_0,r)$, where $X_0=(x_0,t_0)\in\cQ$ and $0<r\le d(x_0)$, $\vec u$ is H\"{o}lder continuous in $\frac{1}{2}Q=Q_{-}(X_0,r/2)$ and we have an estimate
\[
[\vec u]_{\mu_1,\mu_1/2; \frac{1}{2}Q} \le A_1 r^{-\mu_1-(n+2)/2} \norm{\vec u}_{\sL_2(Q)}.
\]
\end{lemma}
\begin{proof}
With the second Korn inequality available to us, the proof is essentially the same as that of \cite[Theorem~3.3]{Kim}.
\end{proof}

Finally, we introduce a condition that was originally considered by Auscher and Tchamitchian \cite{AT2} and is referred to as the \emph{Dirichlet property}.
\begin{enumerate}[{\bf{H}1.}]
\setcounter{enumi}{2}
\item
There exist $\mu_0 \in(0,1]$ and $A_0>0$ such that if $\vec u$ is a weak solution of
\[
\left\{
\begin{array}{lll}
L \vec u = 0  &\text{in}& B \cap \Omega, \\
\vec u=0 & \text{on}& B \cap D,\\
\traction(\vec u) =0  & \text{on} & B \cap N,
\end{array} \right. 
\]
where $B=B(x_0,r)$ with $x_0\in \Omega$ and $0<r\le \diam\Omega$, then for any $0<\rho \le r$, we have
\begin{equation}		\label{dirichlet}
\int_{B(x_0,\rho)\cap \Omega} \abs{D \vec u}^2 \,dx \le A_0 \left(\frac{\rho}{r}\right)^{n-2+2\mu_0} \int_{B(x_0,r)\cap \Omega} \abs{D \vec u}^2 \,dx.
\end{equation}
\end{enumerate}
The following lemma says that H3 implies H2.
Moreover, it shows that if there is $\beta>0$ such that for all $x_0\in\Omega$ and $0<r \le \diam \Omega$, we have
\begin{equation}			\label{typea}
\abs{\Omega\cap B(x_0,r)} \ge \beta r^n,
\end{equation}
then, weak solutions of $\vec u_t -L \vec u =0 $ with homogeneous
boundary data are  H\"{o}lder continuous up to the boundary.
It is not hard to check that $(\epsilon, \delta)$-domains satisfy 
condition \eqref{typea}, so that domains that satisfy H1 also satisfy
(\ref{typea}). 

\begin{lemma}		\label{lem:lhp}
Let $\cQ=\Omega\times [0,\infty)$, $\cD=D\times [0,\infty)$, and
    $\cN=N\times [0,\infty)$. 
If $\Omega$ satisfies the condition \eqref{typea}, then H3 implies 
that  there exist $\mu_1 \in (0,\mu_0)$ and $A_1>0$ such that if $\vec u$ is a weak solution in $\sV_2(Q\cap \cQ)^n$ of 
\begin{equation}			\label{eq2.10tt}
\left\{
\begin{array}{lll}
\vec u_t - L \vec u = 0  &\text{in}& Q \cap \cQ, \\
\vec u=0 & \text{on}& Q \cap \cD,\\
\traction(\vec u) =0  & \text{on} & Q \cap \cN,
\end{array} \right. 
\end{equation}
where $Q=Q_{-}(X_0,r)$ with $X_0=(x_0,t_0) \in\cQ$ and $0<r \le \sqrt{t_0}\wedge \diam \Omega$, 
then $\vec u$ is H\"{o}lder continuous in $\frac{1}{2}Q \cap \cQ$, where  $\frac{1}{2} Q= Q_{-}(X_0,r/2)$, and we have an estimate
\begin{equation}			\label{eq2.11sc}
r^{\mu_1} [\vec u]_{\mu_1,\mu_1/2; \frac{1}{2}Q \cap \cQ}+\abs{\vec u}_{0; \frac{1}{2}Q \cap \cQ} \le A_1 r^{-(n+2)/2} \norm{\vec u}_{\sL_2(Q\cap \cQ)}.
\end{equation}
\end{lemma}
\begin{proof}
See Appendix \ref{proof:lem:lhp}.
\end{proof}

\section{Main theorems}	\label{sec:main}

\begin{theorem}		\label{thm1}
Assume the conditions H1 and H2.
Then there exists a unique heat kernel $\vec K(x,y,t)$ for \eqref{MP}.
It satisfies the symmetry relation
\begin{equation}		\label{eq13.26e}
\vec K(x,y,t)=\vec K(y,x,t)^T
\end{equation}
and thus by \eqref{2.13vv}, for any $\vec f \in \sC^\infty_c(\bar \Omega\times [0,\infty))^n$, we have
\begin{equation}			\label{eq13.37r}
\vec u(x,t)=\int_0^t\!\!\!\int_\Omega \vec K(x,y,t-s) \vec f(y,s)\,dy\,ds
\end{equation}
is, for any $T>0$, a unique weak solution in $\sV^{1,0}_2(\Omega\times(0,T))^n$ of the problem \eqref{2.13vv}.
Also, for $\vec \psi \in L^2(\Omega)^n$, the function $\vec u$ given by
\begin{equation}		\label{eq10.06g}
\vec u(x,t)=\int_\Omega \vec K (x,y,t) \vec \psi(y)\,dy
\end{equation}
is, for any $T>0$, a unique weak solution in $\sV^{1,0}_2(\Omega\times(0,T))^n$ of the problem
\begin{equation}		\label{eq10.08a}
\left\{
\begin{array}{ll}
\vec u_t - L \vec u = 0 \quad & \text{in }\; \Omega \times (0,T), \\
\vec u=0 & \text{on }\;  D\times (0,T),\\
\traction(\vec u) =0  & \text{on }\; N \times (0,T),\\
\vec u(\cdot, 0) = \vec \psi  &  \text{on }\; \Omega
\end{array} \right. 
\end{equation}
and if $\vec \psi$ is continuous at $x_0\in \Omega$ in addition, then
\begin{equation}			\label{eq5.23}
\lim_{\substack{(x,t)\to (x_0,0)\\x\in\Omega,\,t>0}}
\int_{\Omega} \vec K(x,y,t) \vec \psi(y)\,dy =\vec \psi(x_0).
\end{equation}
Moreover, the following estimates holds for all $y\in\Omega$, where we use notation $\cQ=\Omega\times [0,\infty)$, $d_y=d(y)$, and $\hat{Y}=(y,0)$.
\begin{enumerate}[1)]
\item
$\norm{\vec K(\cdot,y,\cdot)}_{\sL_p(Q_{+}(\hat{Y},r))} \le C_p r^{-n+(n+2)/p},\quad \forall r \in (0,d_y],\;\; \forall p \in \bigl[1,\frac{n+2}{n}\bigr)$.

\item
$\abs{\set{(x,t)\in \cQ \colon \abs{\vec K(x,y,t)}>\lambda}} \le C \lambda^{-(n+2)/n}, \quad \forall \lambda> d_y^{-n}$.

\item
$\norm{D_x \vec K(\cdot,y,\cdot)}_{\sL_p(Q_{+}(\hat{Y},r))} \le C_p r^{-n-1+(n+2)/p},\;\; \forall r \in (0,d_y],\;\; \forall p\in \bigl[1,\frac{n+2}{n+1}\bigr)$.

\item
$\abs{\set{(x,t) \in \cQ \colon \abs{D_x \vec K(x,y,t)}>\lambda}} \le C\lambda^{-\frac{n+2}{n+1}}, \quad \forall \lambda> d_y^{-n-1}.$

\item
For $X=(x,t)\in\cQ$ satisfying  $\abs{X-\hat{Y}}_{\sP}<d_y/2$, we have
\begin{equation}		\label{eq3.07c}
\abs{\vec K(x,y,t)}\le C \abs{X-\hat{Y}}_{\sP}^{-n}.
\end{equation}
\item
For $X=(x,t)$ and  $X'=(x',t')$  in $\cQ$ satisfying
\[
2\abs{X'-X}_{\sP}<\abs{X-\hat{Y}}_{\sP}<d_y/2,
\]
we have
\begin{equation}		\label{eq3.08d}
\abs{\vec K(x',y,t')-\vec K(x,y,t)}\le C\abs{X'-X}_{\sP}^{\mu_1}\, \abs{X-\hat{Y}}_{\sP}^{-n-\mu_1},
\end{equation}
\end{enumerate}
In the above, $C$ are constants depending only on $n, \kappa_1,\kappa_2, \Omega, D, \mu_0, A_0$ and $C_p$ depend on $p$ in addition.
\end{theorem}

\begin{remark}			\label{rmk3.7af}
It will be clear from the proof that besides the estimates 1) - 6) in Theorem~\ref{thm1}, we also have
\begin{enumerate}[{\em 1)}]
\setcounter{enumi}{6}
\item
$\norm{\tilde{\vec K}(\cdot,y,\cdot)}_{\sL_{2(n+2)/n}(\cQ \setminus Q_{+}(\hat{Y},r))} \le C r^{-n/2},\quad \forall r \in (0,d_y]$.

\item
$\tri{\tilde{\vec K}(\cdot,y,\cdot)}_{\cQ \setminus Q_{+}(\hat{Y},r)} \le C r^{-n/2},\quad \forall r \in (0,d_y]$.
\end{enumerate}
Here, we use the notation
\[
\tilde{\vec K}(x,y,t)=
\begin{cases} \vec K(x,y,t) & \text{if $D\neq\emptyset$,}
\\
\vec K(x,y,t)- \pi_\cR(\delta_y \vec I)(x) &\text{if $D = \emptyset$.}
\end{cases}
\]
where $\pi_\cR$ is as defined in \eqref{pr}; see also \eqref{eq5.22pj}.
Moreover, if $\Omega$ is such that it admits a bounded linear trace operator from $ W^{1,2}(\Omega)$ to $L^2(\partial\Omega)$, then it can be shown that
for $\vec f \in \sL_{q_1,r_1}(\Omega\times (0,T))^n$ and $\vec g \in \sL_{q_2,r_2}(N\times (0,T))^n$, where $q_k$ and $r_k$ ($k=1,2$) are subject to the conditions of \cite[Theorem~5.1, p. 170]{LSU}, the function $\vec u$ defined by
\[
\vec u(x,t)=\int_{0}^t\!\!\!\int_\Omega \vec K(x,y,t-s) \vec f(y,s)\,dy\,ds+ \int_{0}^t\!\!\!\int_\Omega \vec K(x,y,t-s) \vec g(y,s)\,dS_y\,ds
\]
is a unique weak solution in $\sV^{1,0}_2(\Omega\times (0,T))^n$ of the problem
\[
\left\{
\begin{aligned}
\sL\vec u= \vec f\quad &\text{in }\; \Omega\times (0,T)\\
\vec u= 0 \quad&\text{on }\;  D \times (0,T)\\
\traction(\vec u)= \vec g \quad&\text{on }\; N \times (0,T)\\
\vec u(\cdot, 0)=0 \quad&\text{on }\; \Omega.
\end{aligned}
\right.
\]
The proof is similar to that of the representation formula \eqref{eq13.37r} and is omitted.
\end{remark}

\begin{theorem}			\label{thm3}
Assume H1 and H3. 
There exists the heat kernel $\vec K(x,y,t)$ for \eqref{MP} and it  satisfies all the properties stated in Theorem~\ref{thm1}.
Moreover, for $x, y \in \Omega$ and $t>0$, we have the Gaussian bound
\begin{equation}		\label{eq3.07ktx}
\abs{\vec K(x,y,t)}\le \frac{C}{\left(\sqrt{t}\wedge \diam \Omega \right)^n}\exp\left\{\frac{-\vartheta \abs{x-y}^2}{t}\right\},
\end{equation}
where $C=C(n, \kappa_1, \kappa_2, \Omega, D, \mu_0, A_0)$ and $\vartheta=\vartheta(\kappa_1,\kappa_2)>0$.
Furthermore, for $X=(x,t)$ and $X'=(x',t')$ in $\cQ$ satisfying
$\abs{X'-X}_{\sP} < \frac{1}{2} \left(\abs{X-\hat{Y}}_{\sP} \wedge \diam \Omega\right)$,
we have
\begin{multline}			\label{eq22.00k}
\abs{\vec K(x',y,t')-\vec K(x,y,t)} \le C \left\{\frac{\abs{X'-X}_{\sP}}{\abs{X-\hat{Y}}_{\sP}\wedge \diam \Omega}\right\}^{\mu_1}\\
\times \frac{1}{\left(\sqrt{t}\wedge \diam \Omega\right)^n} \exp\left\{-\frac{\vartheta \abs{x-y}^2}{4t}\right\},
\end{multline}
where $\mu_1$ is as in Lemma~\ref{lem:lhp}.
\end{theorem}

\section{Applications}		\label{sec:apps}
\subsection{Some examples}		\label{sec4.1}
\begin{enumerate}[1.]
\item
If the coefficients are constant, then it is well known that H2 holds with $\mu_0=1$ and $A_0=A_0(n, \kappa_1, \kappa_2)$.
In fact, it is also known that if the coefficients belong to the VMO class, then H2 holds with $\mu_0$ and $A_0$ depending on the BMO modulus of the coefficients as well as on $n, \kappa_1, \kappa_2$.
Therefore, the conclusions of Theorem~\ref{thm1} are valid in these cases.

\item
If the coefficients belong to the VMO class, the domain $\Omega$ is of class $C^1$, and $\bar D \cap \bar N =\emptyset$, then it is known that H3 holds.
Therefore, the conclusions of Theorem~\ref{thm3} are valid in this case.

\item
If $n=2$, then it is well known that H2 holds with $\mu_0=\mu_0(\kappa_1, \kappa_2)$ and $A_0=A_0(\kappa_1, \kappa_2)$.
Therefore, the conclusions of Theorem~\ref{thm1} are valid.
In fact, if $\Omega$ is a Lipschitz domain and $D$ is a (possibly
empty) set satisfying the \emph{corkscrew condition}, i.e.,  for each
$ x\in \partial D$ (where the boundary is taken with respect to
$\partial \Omega$) and $ r \in (0, r_0)$, we may find  $ x_r \in D$ so
that $\abs{x- x_r} \le r$ and $\dist (x_r, \partial \Omega \setminus
D) \ge M^{-1} r$, where $r_0>0$ and $M>0$ are constants, then it is known that H3 holds; see \cite{TKB}. 
Therefore, the conclusions of Theorem~\ref{thm3} are valid in this case.
\end{enumerate}

\subsection{Green's function for the elliptic system}		\label{sec4.2}

We say that an $n\times n$ matrix valued function $\vec G(x,y)$ is the Green's function of $L$ for \eqref{MP} if it satisfies the following properties:
\begin{enumerate}[i)]
\item
$\vec G(\cdot,y) \in W^{1,1}_{loc}(\Omega)$ and $\vec G(\cdot,y) \in W^{1,2}(\Omega\setminus B(y,r))$ for all $y\in\Omega$ and $r>0$.
In the case when $D=\emptyset$, we require $\int_\Omega \vec v(x)^T \vec G(x,y)\,dx =0$ for any $\vec v \in \cR$.
\item
$\vec G(\cdot, y)$ is a weak solution of
\[
-L\vec G(\cdot,y)=\delta_y \vec I - \vec T_y\text{ in }\Omega,\quad \vec G(\cdot,y)=0\text{ on }D,\quad \traction(\vec G(\cdot,y))=0\text{ on }N
\]
in the  sense that we have the identity
\[
\int_{\Omega}a^{\alpha\beta}_{ij} \frac{\partial}{\partial x_\beta} G_{jk}(\cdot,y) \frac{\partial \phi^i}{\partial x_\alpha}\,dx = \phi^k(y).
\]
for any $\vec \phi =(\phi^1,\ldots, \phi^n)^T \in C^\infty(\bar\Omega)^n\cap \vec V$; see \eqref{eq2.04sj} and \eqref{eqv} for the definition of $\vec T_y$ and $\vec V$.

\item
For any $\vec f \in C^\infty_c(\bar\Omega)^n \cap \vec V$, the function $\vec u$ defined by
\begin{equation}	\label{choej}
\vec u(x)=\int_\Omega \vec G(y,x)^T \vec f(y)\,dy
\end{equation}
is the weak solution in $\vec V$ of the problem
\begin{equation}		\label{hee}
\left\{
\begin{array}{lll}
L \vec u = \vec f  &\text{in}& \Omega, \\
\vec u=0 & \text{on}& D,\\
\traction(\vec u) =0  & \text{on} & N.
\end{array} \right. 
\end{equation}
\end{enumerate}
We note that with aid of the second Korn inequality \eqref{eq2.03a}, which is valid for any $\vec u \in \vec V$, the unique solvability of the problem \eqref{hee} in the space $\vec V$ is an immediate consequence of Lax-Milgram lemma. 
We also note that the property iii) of the above definition together with the requirement
\[
\int_\Omega \vec v(x)^T \vec G(x,y)\,dx =0,\quad \forall \vec v \in \cR
\]
gives the uniqueness of the Green's function.

\begin{theorem}		\label{thm4}
Assume the conditions H1 and H2.
Then, there exists a unique  Green's function $\vec G(x,y)$ for \eqref{MP}.
We have
\begin{equation}
\label{eq:E23}
\vec G(y,x)= \vec G(x,y)^T
\end{equation}
and thus by \eqref{choej}, for any $\vec f \in C^\infty(\bar \Omega)^n \cap \vec V$, we find
\[
\vec u(x)=\int_\Omega \vec G(x,y) \vec f(y)\,dy
\]
is a unique weak solution in $\vec V$ of the problem \eqref{hee}.
If we assume H3 instead of H2, then for any $x, y \in \Omega$, we have
\begin{enumerate}[i)]
\item
$n=2$
\begin{equation}						\label{eq13.03k}
\abs{\vec G(x,y)} \le C\left\{1+\ln \left(\frac{\diam \Omega}{\abs{x-y}}\right)\right\},
\end{equation}

\item
$n\ge 3$
\begin{equation}						\label{eq13.03l}
\abs{\vec G(x,y)} \le C\abs{x-y}^{2-n},
\end{equation}
\end{enumerate}
and moreover, for any $x,y \in\Omega$ satisfying $\abs{x-x'}<\frac{1}{2}\abs{x-y}$, we have
\begin{equation}				\label{eq15.13g}
\abs{\vec G(x',y)-\vec G(x,y)} \le C  \abs{x'-x}^{\mu_1} \abs{x-y}^{2-n-\mu_1}.
\end{equation}
In the above, $C$ is a constant depending on the prescribed parameters and $\diam \Omega$ and $\mu_1 \in (0,1)$ is as in Lemma~\ref{lem:lhp}.
\end{theorem}

\begin{corollary}
Let $n=2$ and assume H1.
Then, there exists the Green's function for \eqref{MP} that satisfies \eqref{eq:E23}.
Moreover, if $\Omega$ is a Lipschitz domain and $D$ satisfies the corkscrew condition described in Section~\ref{sec4.1}, then the estimates \eqref{eq13.03k}, \eqref{eq15.13g} hold, and $\mu_1$ and $C$ are constants determined by  $\kappa_1$, $\kappa_2$, $\Omega$, and $D$ .
\end{corollary}
\begin{proof}
Follows from Example~3 in Section~\ref{sec4.1} and Theorem~\ref{thm4}.
\end{proof}

\section{Proofs of main theorems}		\label{sec:pf}
\subsection{Proof of Theorem~\ref{thm1}}
In the proof, we denote by $C$ a constant depending on the prescribed parameters $n, \kappa_1, \kappa_2, \mu_0, A_0$ as well as on $\Omega$ and $D$; if it depends also on some other parameters such as $p$, it will be written as $C_p$, etc.

We fix a $\Phi \in C_c^\infty(\bR^{n})$ such that $\Phi$ is supported in $B(0,1)$, $0\le \Phi \le 2$, and $\int_{\bR^{n}} \Phi=1$.
Let $y\in \Omega$ be fixed but arbitrary.
For $\epsilon>0$, we define
\[
\Phi_{y,\epsilon}(x)=\epsilon^{-n}\Phi((x-y)/\epsilon)
\]
and let $\vec v_\epsilon=\vec v_{\epsilon, y,k}$ be a unique weak solution in $\sV^{1,0}_2(\Omega\times(0,T))^n$ of the problem
\begin{equation} \label{eq10.03a}
\left\{
\begin{array}{ll}
\vec u_t- L \vec u = 0  & \text{in }\; \Omega \times (0,T),\\
\vec u=0 & \text{on }\;  D\times (0,T),\\
\traction(\vec u) =0 & \text{on }\; N \times (0,T),\\
\vec u(\cdot, 0) = \Phi_{y,\epsilon} \vec e_k  &  \text{on }\; \Omega,
\end{array}
\right. 
\end{equation}
where $\vec e_k$  is the $k$-th unit column vector in $\bR^n$; see Lemma~\ref{lem:pre}.
By the uniqueness, we find that $\vec v_\epsilon$ does not depend on a particular choice of $T$ and thus by setting $\vec v_\epsilon(x,t) = 0$ for $t<0$ and letting $T\to \infty$, we may assume the $\vec v_\epsilon$ is defined on the entire $\Omega\times (-\infty,\infty)$.
We define \emph{the mollified heat kernel} $\vec K^\epsilon(x,y,t)$ to be an $n\times n$ matrix valued function whose $k$-th column is $\vec v_{\epsilon,y,k}(x,t)$; i.e.,
\[
K^\epsilon_{jk}(x,y,t)=v^j_\epsilon(x,t)=v^j_{\epsilon, y,k}(x,t).
\]
For $\vec f \in \sC^\infty_c(\bar \Omega\times (-\infty,\infty))^n$, fix $a, b$ so that $a<0<b$ and $\supp \vec f \subset \bar\Omega\times (a,b)$.
Let $\vec u$ be a weak solution in $\sV_2^{1,0}(\Omega\times(a,b))^n$ of the backward problem
\[
\left\{\begin{array}{ll}
-\vec u_t - L \vec u =\vec f \quad & \text{in }\; \Omega \times (a,b)\\
\vec u=0 & \text{on }\;  D\times (a,b)\\
\traction(\vec u) =0  & \text{on }\; N \times (a,b)\\
\vec u(\cdot, b) = 0 &  \text{on }\; \Omega.
\end{array}
\right. 
\]
Then, it is easy to see that we have
\begin{equation}		\label{eq11.05b}
\int_\Omega \Phi_{y,\epsilon}(x) u^k(x,0)\,dx=\int_0^b\!\!\! \int_\Omega K^\epsilon_{ik}(x,y,t) f^i(x,t)\,dx\,dt.
\end{equation}
Next, we define $\tilde{\vec K}{}^\epsilon(x,y,t)$ by
\begin{equation}	\label{eq5.09ck}
\tilde{\vec K}{}^\epsilon(x,y,t) :=
\begin{cases} \vec K^\epsilon(x,y,t) & \text{if $D\neq\emptyset$,}
\\
\vec K^\epsilon(x,y,t)-1_{[0,\infty)}(t) \,\pi_\cR(\Phi_{y,\epsilon} \vec I)(x) &\text{if $D = \emptyset$,}
\end{cases}
\end{equation}
where $\pi_\cR(\Phi_{y,\epsilon} \vec I)(x)$ is an $n\times n$ matrix whose $k$-th column is $\pi_\cR(\Phi_{y,\epsilon} \vec e_k)(x)$.
We set $\tilde{\vec v}_\epsilon=\tilde{\vec v}_{\epsilon,y,k}$ to be the $k$-th column of $\tilde{\vec K}{}^\epsilon(\cdot,y,\cdot)$.
It is easy to verify that $\tilde{\vec v}_\epsilon(\cdot,t) \in \vec V$ for a.e. $t>0$.
Therefore, for any $T>0$, it is the weak solution of the problem (see Section~\ref{sec:weaksol} and Lemma~\ref{lem:pre})
\begin{equation} \label{eq10.03ap}
\left\{
\begin{array}{ll}
\vec u_t- L \vec u = 0  & \text{in }\; \Omega \times (0,T)\\
\vec u=0 & \text{on }\;  D\times (0,T)\\
\traction(\vec u) =0 & \text{on }\; N \times (0,T)\\
\vec u(\cdot, 0) = \vec \psi_{\epsilon, y,k}  &  \text{on }\; \Omega\\
\vec u(\cdot,t) \in \vec V& \text{for a.e.}\;  t\in (0,T),
\end{array}
\right. 
\end{equation}
where we denote
\[
\vec \psi_{\epsilon,y,k}=
\begin{cases} \Phi_{\epsilon}\,\vec e_k & \text{if $D\neq \emptyset$,}
\\
\Phi_{y,\epsilon}\,\vec e_k -\pi_\cR(\Phi_{y,\epsilon}\,\vec e_k) &\text{if $D = \emptyset$}.
\end{cases}
\]
By the energy inequality, we get (see Lemma~\ref{lem:pre})
\begin{equation}		\label{eq11.03ap}
\tri{\tilde{\vec v}_{\epsilon}}_{\Omega\times (-\infty,\infty)}\le  C \norm{\vec \psi_{\epsilon,y,k}}_{L^2(\Omega)} \le C \epsilon^{-n/2}.
\end{equation}
Let $\vec f$ be a smooth function supported in $Q_{+}(X_0,R) \subset \Omega\times (-\infty,\infty)$.
Fix $b>t_0+R^2$ and let $\tilde{\vec u}$ be the weak solution of the backward problem
\begin{equation}			\label{eq5.06jj}
\left\{\begin{array}{ll}
-\vec u_t - L \vec u =\vec f \quad & \text{in }\; \Omega \times (a,b)\\
\vec u=0 & \text{on }\;  D\times (a,b)\\
\traction(\vec u) =0  & \text{on }\; N \times (a,b)\\
\vec u(\cdot, b) = 0 &  \text{on }\; \Omega\\
\vec u(\cdot,t) \in \vec V& \text{for a.e.}\;  t\in (a,b).
\end{array}
\right. 
\end{equation}
The unique solvability of the above problem is similar to Lemma~\ref{lem:pre} and by setting $\tilde{\vec u}(x,t)=0$ for $t>b$ and letting $a\to -\infty$, we may again assume that $\tilde{\vec u}$ is defined on $\Omega\times (-\infty,\infty)$.
Then, similar to \eqref{enieq}, we have
\begin{equation}		\label{eq11.06b}
\tri{\tilde{\vec u}}_{\Omega\times (-\infty,\infty)}\le C\norm{\vec f}_{\sL_{2(n+2)/(n+4)}(Q_{+}(X_0,R))}.
\end{equation}
By using H\"{o}lder's inequality, we derive from \eqref{eq11.06b} that
\[
\norm{\tilde{\vec u}}_{\sL_2(Q_{+}(X_0,R))}\le C R^{3+n/2} \norm{\vec f}_{\sL_{\infty}(Q_{+}(X_0,R))}.
\]
Then by Lemma~\ref{lem3a} below and the above estimate, we obtain
\begin{equation}			\label{eq11.07b}
\abs{\tilde{\vec u}}_{0; Q_{+}(X_0,R/2)} \le C R^{2}\norm{\vec f}_{\sL_\infty(Q_{+}(X_0,R))}.
\end{equation}

\begin{lemma}			\label{lem3a}
H2 implies that $\tilde{\vec u}$ is continuous in $Q_{+}(X_0,R/2)$ and satisfies the estimate
\begin{equation}				\label{eq16.01f}
\abs{\tilde{\vec u}}_{0; \frac{1}{2}Q} \le C\left(R^{-(n+2)/2} \norm{\tilde{\vec u}}_{\sL_2(Q)}+ R^{2}\norm{\vec f}_{\sL_\infty(Q)}\right),
\end{equation}
where we denote $\alpha Q=Q_{+}(X_0,\alpha R)$.
In fact, the same conclusion is true if $\tilde{\vec u}$ is a weak solution in $\sV_2(Q)^n$ of $-\vec u_t -L \vec u=\vec f$  (or $\vec u_t - L \vec u = \vec f$) with  $\vec f \in \sL_\infty(Q)^n$.
\end{lemma}
\begin{proof}
See Appendix \ref{proof:lem3a}. 
\end{proof}

Note that similar to \eqref{eq11.05b}, we have the identity
\begin{equation}		\label{eq11.05bt}
\int_\Omega \Phi_{y,\epsilon}(x) \tilde u^k(x,0)\,dx=\int_{-\infty}^\infty \int_\Omega \tilde K^\epsilon_{ik}(\cdot,y,\cdot) f^i\,dX.
\end{equation}
If $B(y,\epsilon)\times \set{0}\subset Q_{+}(X_0,R/2)$, then \eqref{eq11.05bt} together with \eqref{eq11.07b} yields
\[
\Abs{\iint_{Q_{+}(X_0,R)}\!\!\! \tilde K^\epsilon_{ik}(\cdot,y,\cdot) f^i\,dX} \le
 \abs{\tilde{\vec u}}_{0; Q_{+}(X_0,R/2)} \le CR^{2}\norm{\vec f}_{\sL_{\infty}(Q_{+}(X_0,R))}.
\]
Therefore, by duality, it follows that we have
\begin{equation}					\label{eq11.09a}
\norm{\tilde{\vec K}{}^\epsilon(\cdot,y,\cdot)}_{\sL_1(Q_{+}(X_0,R))}\le CR^{2}
\end{equation}
provided $0<R<d_y$ and $B(y,\epsilon)\times \set{0}\subset Q_{+}(X_0, R/2)$.
For $X\in \cQ $ such that $0<d:= \abs{X-\hat{Y}}_\sP< d_y/6$, if we set $r=d/3$, $X_0=(y, -2d^2)$, and $R=6d$, then it is easy to see that for $\epsilon<d/3$, we have
\[
B(y,\epsilon)\times\set{0}\subset Q_{+}(X_0,R/2), \quad Q_{-}(X,r)\subset Q_{+}(X_0,R),
\]
and also that $\tilde{\vec v}_\epsilon=\tilde{\vec v}_{\epsilon,y,k}$ is a weak solution in $\sV_2(Q_{-}(X,r))$ of $\vec u_t-L \vec u=0$.

\begin{lemma}		\label{lem3c}
H2 implies that
for any $p>0$, we have
\[
\abs{\vec u}_{0; Q_{-}(X,r/2)} \le C_p r^{-(n+2)/p} \norm{\vec u}_{\sL_p(Q_{-}(X,r))}.
\]
\end{lemma}
\begin{proof}
It follows from the estimate \eqref{eq16.01f} in Lemma~\ref{lem3a} together with a standard argument described in \cite[pp. 80--82]{Gi93}.
\end{proof}

Note that by Lemma~\ref{lem3c} and \eqref{eq11.09a}, we obtain
\[
\abs{\tilde{\vec v}_\epsilon(X)}\le Cr^{-n-2} \norm{\tilde{\vec v}_\epsilon}_{\sL_1(Q_{-}(X,r))}
\leq Cr^{-n-2}\norm{\tilde{\vec v}_\epsilon}_{\sL_1(Q_{+}(X_0,R))} \le C d^{-n},
\]
That is, for $X=(x,t)\in \cQ$ satisfying $0<\abs{X-\hat{Y}}_\sP<d_y/6$, we have 
\begin{equation}					\label{eq11.16a}
\abs{\tilde{\vec K}{}^\epsilon(x,y,t)} \le  C \abs{X-\hat{Y}}_{\sP}^{-n},\quad \forall \epsilon \le \tfrac{1}{3}\abs{X-\hat{Y}}_{\sP}.
\end{equation}
Next, we claim that for $0<R \le d_y$, we have
\begin{equation}			\label{eq11.19a}
\tri{\tilde{\vec K}{}^\epsilon(\cdot,y,\cdot)}_{\cQ \setminus  Q_{+}(\hat{Y},R)} \le C R^{-n/2}, \quad \forall \epsilon>0.
\end{equation}
To prove \eqref{eq11.19a}, we only need to consider the case when $R>6 \epsilon$.
Indeed, if $R\le 6 \epsilon$, then \eqref{eq11.03ap} yields
\[
\tri{\tilde{\vec K}{}^\epsilon(\cdot,y,\cdot)}_{\cQ \setminus  Q_{+}(\hat{Y},R)}
\le \tri{\tilde{\vec K}{}^\epsilon(\cdot,y,\cdot)}_{\Omega\times (-\infty,\infty)} \le C\epsilon^{-n/2} \le CR^{-n/2}.
\]
Fix a cut-off function $\zeta\in \sC^\infty_c(Q(\hat{Y},R))$ such that
\begin{equation}			\label{wse.eq3aa}
\zeta\equiv 1\;\text{ on }\; Q_{+}(\hat{Y}, R/2),\quad
0\le \zeta \le 1,\quad \abs{D_x \zeta}\le 4 R^{-1}, \quad \abs{\zeta_t}\le 16 R^{-2}.
\end{equation}
By using the second Korn inequality \eqref{eq2.08q} and \eqref{eq11.16a}, we derive from \eqref{eq10.03ap} that
\begin{align}
\nonumber
\sup_{t \ge 0}\int_{\Omega} \abs{(1-\zeta) \tilde{\vec v}_\epsilon&(x,t)}^2\,dx+\iint_{\cQ} \abs{D_x ((1-\zeta) \tilde{\vec v}_\epsilon)}^2 \,dx dt\\
\nonumber
&\le C \iint_{\cQ} \left(\abs{D_x \zeta}^2+\abs{(1-\zeta)\zeta_t}\right) \abs{\tilde{\vec v}_\epsilon}^2 \,dx dt\\
&\le C R^{-2} \iint_{\set{R/2<\abs{X-\hat{Y}}_{\sP}<R}} \abs{X-\hat{Y}}_{\sP}^{-2n}\,dX \le CR^{-n},
			\label{eq11.20a}
\end{align}
which implies the desired estimate \eqref{eq11.19a}.
In fact, we obtain from \eqref{eq11.20a} that 
\begin{equation}		\label{eq11.21a}
\tri{(1-\zeta) \tilde{\vec K}{}^\epsilon(\cdot,y,\cdot)}_{\cQ} \le C R^{-n/2},
\quad\forall\epsilon>0.
\end{equation}
We claim that for $0<R \le d_y$, we have
\begin{equation}			\label{eq5.11sj}
\norm{\tilde{\vec K}{}^\epsilon(\cdot,y,\cdot)}_{\sL_{2(n+2)/n}(\cQ\setminus \bar Q_{+}(\hat{Y},R))}\le C R^{-n/2}.
\end{equation}
Indeed, set $\cQ_{(1)}:= \Omega\times(R^2,\infty)$ and $\cQ_{(2)}:= (\Omega\setminus \bar B(y,R))\times(0,R^2)$ and note that
by \eqref{eq5.13d} and \eqref{eq11.19a} we have
\[
\norm{\tilde{\vec K}{}^\epsilon(\cdot,y,\cdot)}_{\sL_{2(n+2)/n}(\cQ_{(1)})} \le C \gamma \tri{\tilde{\vec K}{}^\epsilon(\cdot,y,\cdot)}_{\cQ_{(1)}}
\le C\gamma R^{-n/2}
\]
and similarly, by \eqref{paraemb} and \eqref{eq11.21a}, we have
\[
\norm{\tilde{\vec K}{}^\epsilon(\cdot,y,\cdot)}_{\sL_{2(n+2)/n}(\cQ_{(2)})}
\le C (2\gamma+1)^{\frac{1}{n+2}} \gamma R^{-n/2}.
\]
By combining the above two inequalities, we get \eqref{eq5.11sj}.
\begin{lemma}		\label{wse.lem1}
For any $y\in \Omega$ and $\epsilon>0$, we have
\begin{align}
\label{wse.eq2e}
\abs{\set{X=(x,t)\in\cQ \colon \abs{\tilde{\vec K}{}^\epsilon(x,y,t)}&>\lambda}} \le C\lambda^{-\frac{n+2}{n}}, \quad \forall \lambda>d_y^{-n},\\
\label{eq2.15f}
\abs{\set{X=(x,t) \in \cQ \colon \abs{D_x \tilde{\vec K}{}^\epsilon(x,y,t)}&>\lambda}} \le C \lambda^{-\frac{n+2}{n+1}},\quad \forall \lambda> d_y^{-(n+1)}.
\end{align}
Also, for any $y\in \Omega$, $0<R\le d_y$, and $\epsilon>0$, we have
\begin{align}
\label{wse.eq2h}
\norm{\tilde{\vec K}{}^\epsilon(\cdot,y,\cdot)}_{\sL_p(Q_{+}(\hat{Y},R))}&\le C_{p} R^{-n+(n+2)/p},  \quad \forall p\in [1,\tfrac{n+2}{n}),\\
\label{wse.eq2i}
\norm{D_x \tilde{\vec K}{}^\epsilon(\cdot,y,\cdot)}_{\sL_p(Q_{+}(\hat{Y},R))} &\le C_{p} R^{-n-1+(n+2)/p},  \quad \forall p\in[1,\tfrac{n+2}{n+1}).
\end{align}
\end{lemma}
\begin{proof}
We derive \eqref{wse.eq2h} and \eqref{wse.eq2i}, respectively, from
\eqref{wse.eq2e} and  \eqref{eq2.15f}, which in turn follow  from
\eqref{eq5.11sj} and \eqref{eq11.19a}, respectively; see \cite[Lemmas~3.3 and 3.4]{CDK}.
\end{proof}

\begin{lemma}	\label{lem5.18ah}
Suppose $\set{\vec u_k}_{k=1}^\infty$ is a sequence in
$\sV_2(\cQ)^n$ such that $\sup_k \tri{\vec u_k}_{\cQ}\le A<\infty$,
then there exists $\vec u \in \sV_2(\cQ)^n$ satisfying $\tri{\vec
  u}_{\cQ}\le A$ and a subsequence $\vec u_{k_j}$ that converges to
$\vec u$ weakly in  $\sW^{1,0}_2(\Omega \times (0,T))^n$ for any
$T>0$.
Moreover, if each $\vec u_k(\cdot,t) \in \vec V$ for a.e. $t\in
(0,\infty)$, then we also have   $\vec u(\cdot, t) \in \vec V $ for
a.e. $t\in (0, \infty)$.
\end{lemma}
\begin{proof}
See \cite[Lemma~A.1]{CDK}.
\end{proof}
The above two lemmas contain all the ingredients for the construction of a function $\tilde{\vec K}(\cdot,y,\cdot)$ such that for a sequence $\epsilon_\mu$ tending to zero, we have
\begin{align}
\nonumber
\tilde{\vec K}{}^{\epsilon_\mu}(\cdot,y,\cdot) &\rightharpoonup \tilde{\vec K}(\cdot,y,\cdot) \;\text{ weakly in }\, \sW^{1,0}_q(Q_{+}(\hat{Y},d_y))^{n^2},\\
\label{eq5.27a}
(1-\zeta)\tilde{\vec K}{}^{\epsilon_\mu}(\cdot,y,\cdot) &\rightharpoonup (1-\zeta)\tilde{\vec K}(\cdot,y,\cdot) \; \text{ weakly in }\,\sW^{1,0}_2(\Omega\times (0,T))^{n^2},
\end{align}
where $1<q<\frac{n+2}{n+1}$, $\zeta$ is as in \eqref{wse.eq3aa} with $R=\bar d_Y/2$, and $T>0$ is arbitrary.
It is routine to check that $\tilde{\vec K}(\cdot,y,\cdot)$ satisfies the same estimates as in Lemma~\ref{wse.lem1} as well as \eqref{eq11.19a} and \eqref{eq5.11sj}; see \cite[Section~4.2]{CDK}.
Note that by Lemma~\ref{lem5.18ah} we have $\tilde{\vec K}(\cdot,t,y) \in \vec V$ for a.e. $t>0$.
We define $\vec{K}(x,y,t)$ by
\begin{equation}	\label{ntilde}
\vec K(x,y,t):=
\begin{cases} \tilde{\vec K}(x,y,t) & \text{if $D\neq\emptyset$,}
\\
\tilde{\vec K}(x,y,t)+1_{[0,\infty)}(t) \,\pi_\cR(\delta_y \vec I)(x) &\text{if $D = \emptyset$.}
\end{cases}
\end{equation}
where $\pi_\cR(\delta_y \vec I)$ is an $n\times n$ matrix valued function whose $k$-th column is
\begin{multline}			\label{eq5.22pj}
\pi_\cR(\delta_y \vec e_k)=\sum_{i=1}^{N}\, \omega_i^k(y) \vec \omega_i,\\
\text{where }\;\vec \omega_i=(\omega_i^1,\ldots, \omega_i^n)^T \in \cR\;\text{ and }\; \int_\Omega \vec \omega_i\cdot \vec \omega_j\,dx = \delta_{ij}. 
\end{multline}
Then, it is easy to see that $\vec K(x,y,t)$ satisfies the estimates
1) - 4) in Theorem \ref{thm1} because $\tilde{\vec K}(x,y,t)$ satisfies all of them as we noted above, and $\cR$ is a finite dimensional vector space so that all norms over $\cR$ are equivalent.
For example, to see the estimate 3) in the theorem holds, observe that
\begin{multline*}
\norm{D\vec \omega_i}_{\sL^p(Q_{+}(\hat{Y},r))} \le r \norm{\vec \omega_i}_{W^{1,2}(B(y,r))} \abs{Q_{+}(\hat{Y}, r)}^{\frac{1}{p}-\frac{1}{2}}\\
 \le C r^{-n/2+(n+2)/p} \le C (\diam \Omega)^{n/2+1} r^{-n-1+(n+2)/p}.
\end{multline*}
Since $\pi_R(\delta_y \vec e_k)\in \cR$, it is a weak solution of
$\vec u_t -L \vec u=0$.
Therefore, by repeating the  proof for \eqref{eq11.16a}, we find $\vec K(x,y,t)$ satisfies the estimate \eqref{eq3.07c}, while the estimate \eqref{eq3.08d} is obtained from \eqref{eq3.07c} and Lemma~\ref{lem:ihp}.
We have thus shown that $\vec K(x,y,t)$ satisfies all the estimates 1)
- 6) in  Theorem \ref{thm1}.

We now prove that $\vec K(x,y,t)$ satisfies all the properties stated in Section~\ref{sec2.4} so that it is indeed the heat kernel for \eqref{MP}.
First, note that the property a) is clear from 1), 3) in the theorem and 8) in Remark~\ref{rmk3.7af}.
To verify the property b), first note that by \eqref{eq5.09ck}  together with \eqref{eq5.27a} and \eqref{ntilde}, we have
\begin{align}
\nonumber
\vec K^{\epsilon_\mu}(\cdot,y,\cdot) &\rightharpoonup \vec K(\cdot,y,\cdot) \;\text{ weakly in }\, \sW^{1,0}_q(Q_{+}(\hat{Y},d_y))^{n^2},\\
\label{eq5.28bt}
(1-\zeta)\vec K^{\epsilon_\mu}(\cdot,y,\cdot) &\rightharpoonup (1-\zeta)\vec K(\cdot,y,\cdot) \; \text{ weakly in }\,\sW^{1,0}_2(\Omega\times (0,T))^{n^2}, 
\end{align}
for any $T>0$.
Next, suppose $\vec \phi=(\phi^1,\ldots,\phi^n)^T$ is supported in $\bar\Omega \times (0,T)$ and note that by \eqref{eq10.03a} we have
\begin{multline*}
\int_\Omega \Phi_{y, \epsilon}(x) \phi^k(x,0)\,dx= \int_0^T\!\!\!\int_{\Omega} - K^{\epsilon_\mu}_{ik}(x,y,t)\frac{\partial}{\partial t} \phi^i(x,t)\,dx\,dt\\
+\int_0^T\!\!\!\int_{\Omega} a^{\alpha\beta}_{ij} \frac{\partial}{\partial x_\beta}   K^{\epsilon_\mu}_{jk}(x,y,t) \frac{\partial}{\partial x_\alpha} \phi^i(x,t) \,dx\,dt.
\end{multline*}
By writing $\vec \phi= \eta \vec \phi + (1-\eta)\vec \phi$,  where $\eta \in \sC^\infty_c(Q(\hat{Y}, d_y))$ satisfying $\eta =1$ on $Q(\hat{Y}, d_y/2)$, and using \eqref{eq5.28bt}, and taking $\mu\to \infty$ in the above, we get the identity \eqref{eq2.04x}; see \cite[p. 1662]{CDK} for the details.
To verify the property c), let us denote $\hat{\vec f}(x,t)=\vec f(x,-t)$ and let $\hat{\vec u}$ be a unique weak solution in $\sV^{1,0}_2(\Omega\times (-T,0))^n$ of the backward problem
\[
\left\{
\begin{array}{ll}
-\vec u_t - L \vec u = \hat{\vec f} \quad & \text{in }\; \Omega \times (-T,0)\\
\vec u=0 & \text{on }\;  D\times (-T,0)\\
\traction(\vec u) =0  & \text{on }\; N \times (-T,0)\\
\vec u(\cdot, 0) = 0  &  \text{on }\; \Omega.
\end{array} \right. 
\]
By letting $T\to \infty$, we may assume that $\hat{\vec u}$ is defined on $\Omega\times (-\infty,0)$.
Then, similar to \eqref{eq11.05b}, for $t>0$, we have
\[
\int_\Omega \Phi_{x,\epsilon}(y) \hat u^k(y,-t)\,dy=\int_{-t}^0 \int_\Omega K^\epsilon_{ik}(y,x,s+t) \hat f^i(y,s)\,dy\,ds.
\]
We note H2 implies, similar to Lemma~\ref{lem3a}, that $\hat{\vec u}$ is continuous in $\Omega\times (-\infty,0)$.
By writing $\vec f=\zeta \vec f+ (1-\zeta) \vec f$ and use \eqref{eq5.28bt} to get
\[
\hat u^k(x,-t)=\int_{0}^{t} \!\!\! \int_\Omega K_{ik}(y,x,t-s) f^i(y,s)\,dx\,ds
\]
If we set $\vec u(x,t)=\hat{\vec u}(x,-t)$, then it becomes be a weak solution in $\sV_2(\Omega\times (0,T))^n$ of the problem \eqref{2.13vv}, and thus by the uniqueness the property c) is confirmed.
Therefore, we have shown that $\vec K(x,y,t)$ is indeed the heat kernel for \eqref{MP}.

Now, we prove the identity \eqref{eq13.26e}.
Let
\begin{equation}			\label{eq13.27c}
\hat{K}^\delta_{il}(y,x,s)=\hat{v}^i_{\delta,x,l}(y,s)=v_{\delta,x,l}^i(y,t-s)=K^\delta_{il}(y,x,t-s)
\end{equation}
and
\[
\hat{\Phi}_{x,\delta}(y,s)=\Phi_{x,\delta}(y,t-s),
\]
where $\vec v_{\delta,x,l}$ and $\Phi_{x,\delta}$ are as above.
Observe that $\hat{\vec v}_{\delta,x,l}(y,s)$ is, for any $-T<t$, a unique weak solution in $\sV^{1,2}_0(\Omega\times (-T,t))^n$ of the problem
\begin{equation}			\label{eq3.31e}
\left\{
\begin{array}{ll}
-\vec v_s- L \vec v =0 \; & \text{in }\; \Omega \times (-T, t)\\
\vec v=0 & \text{on }\;  D\times (-T, t)\\
\traction(\vec v) =0 & \text{on }\; N \times (-T, t)\\
\vec v(\cdot, t) = \hat{\Phi}_{x,\delta}\, \vec e_l &  \text{on }\; \Omega.
\end{array}
\right. 
\end{equation}
Then, similar to \eqref{eq11.05b}, we have
\begin{equation}			\label{eq3.31w}
\int_{\Omega} \hat{K}^\delta_{kl}(\cdot,x,0) \Phi_{y,\epsilon} =\int_\Omega K^\epsilon_{lk}(\cdot,y,t) \hat{\Phi}_{x,\delta}.
\end{equation}
By repeating the proof of \cite[Lemma~3.5]{CDK}, we obtain \eqref{eq13.26e} from \eqref{eq3.31w} as well as the the following representation of the mollified heat kernel:
\[
\vec K^\epsilon(x,y,t)=\int_\Omega \vec K(z,x,t)^T \Phi_{y,\epsilon}(z)\,dz.
\]
In particular, by continuity of $\vec K(\cdot,x,t)$ and \eqref{eq13.26e}, we have 
\begin{equation}		\label{eq3.786}
\lim_{\epsilon \to 0} \vec K^\epsilon(x,y,t)=\vec K(x,y,t).
\end{equation}

Now, we turn to the proof of the formula \eqref{eq10.06g}.
Let $\vec u$ be the weak solution in $\sV^{1,0}_2(\Omega\times(0,T))^n$ of the problem \eqref{eq10.08a}.
Let $X=(x,t)\in \Omega\times (0,T)$ be fixed but arbitrary and let $\hat{\vec v}_\delta=\hat{\vec v}_{\delta,x,l}$ be as in \eqref{eq13.27c}.
Then, it follows from the equations \eqref{eq10.08a} and \eqref{eq3.31e} that  for sufficiently small $\delta$, we have
\[
\int_\Omega \psi^i(y) \hat{v}_\delta^i(y,0)\,dy=\int_\Omega u^l(y,t) \hat \Phi_{\delta,x}(y)\,dy.
\]
Therefore, by using \eqref{eq13.27c}, we obtain
\begin{equation}	\label{eqn:3.56}
\int_\Omega  K^{\epsilon_\mu}_{il}(y,x,t) \psi^i(y)\,dy=\int_\Omega u^l(y,t) \hat \Phi_{\epsilon_\mu,x}(y)\,dy.
\end{equation}
By \eqref{eq11.19a} and \eqref{eq3.786} with $x$ in place of $y$, we find by the dominated convergence theorem that
\[
\lim_{\mu\to \infty} \int_\Omega  K^{\epsilon_\mu}_{il}(y,x,t) \psi^i(y)\,dy=\int_\Omega  K_{il}(y,x,t) \psi^i(y)\,dy.
\]
By Lemma~\ref{lem3a}, we find that $\vec u$ is continuous at $X=(x,t)$.
Therefore, by taking the limit $\mu\to\infty$ in \eqref{eqn:3.56} and
using \eqref{eq13.26e}  we obtain \eqref{eq10.06g}.

Finally, let $\vec u$ be a weak solution $\sV^{1,0}_2(\Omega\times(0,T))^n$ of the problem \eqref{eq10.08a} and $\phi$ be a Lipschitz function on $\Omega$ satisfying $\abs{\nabla \phi} \le K$ a.e. for some $K>0$.
Denote
\[
I(t):= \int_{\Omega} e^{2\phi}\abs{\vec u(x,t)}^2\,dx.
\]
Then $I'(t)$ satisfies for a.e. $t>0$ the differential inequality
\begin{align}
\nonumber
I'(t) & =-2\int_\Omega \left\{e^{2\phi} \sB(\vec u,\vec u)+2e^{2\phi} a^{\alpha\beta}_{ij} \frac{\partial u^j}{\partial x_\beta} \frac{\partial \phi}{\partial x_\alpha} u^i \right\}\,dx\\
\nonumber
& \le \int_\Omega \left\{-2 \kappa_1 e^{2\phi}\abs{\strain(\vec u)}^2 + 4\kappa_2 e^{2 \psi} \abs{\strain(\vec u)} \abs{\nabla \phi}\abs{\vec u} \right\}\,dx\\
\nonumber
& \le  \int_\Omega \left\{-2\kappa_1 e^{2\phi} \abs{\strain(\vec u)}^2
+ 2\kappa_1 e^{2 \phi} \abs{\strain(\vec u)}^2+ 2(\kappa_2^2/\kappa_1) K^2 e^{2\phi} \abs{\vec u}^2 \right\} \,dx\\
\label{eq10.17jw}
& \le 2(\kappa_2^2/\kappa_1) K^2 I(t).
\end{align}
Then, by repeating the argument in \cite[Section~4.4]{CDK}, we obtain the formula \eqref{eq5.23}.
The theorem is proved.
\hfill\qedsymbol

\subsection{Proof of Theorem~\ref{thm3}}
By Lemma~\ref{lem:lhp} and a remark preceding it, we observe that the conditions H1 and H3 imply the condition H2 and the condition (LB) in \cite{CDK2}.
Then, by Theorem~\ref{thm1}, the heat kernel $\vec K(x,y,t)$ exists and by using \eqref{eq10.17jw} and repeating the proof of \cite[Theorem~3.1]{CDK2} with $R_{max}=\diam \Omega$, we obtain the Gaussian bound \eqref{eq3.07ktx}.
Also, by \eqref{eq2.11sc}, for $X=(x,t)$ and $X'=(x',t')$ in $\cQ$ satisfying
\[
2\abs{X'-X}_{\sP} <r \le r_0  := \abs{X-\hat{Y}}_{\sP} \wedge \diam \Omega,
\]
we have
\begin{equation}		\label{eq17.24}
\abs{\vec K(x',y,t')-\vec K(x,y,t)}
\le C \abs{X'-X}_{\sP}^{\mu_1}\, r^{-\mu_1-(n+2)/2} \norm{\vec K(\cdot,y,\cdot)}_{\sL_2(Q_{-}(X,r)\cap \cQ)}
\end{equation}
Note that the estimate \eqref{eq3.07ktx} implies that for $0<s \le (\diam \Omega)^2$, we have
\begin{equation}					\label{eq5.18b}
\abs{\vec K(z,y,s)} \le C \left\{\abs{z-y} \wedge \sqrt{s}\right\}^{-n}.
\end{equation}
From the above the estimate, we obtain the estimate \eqref{eq22.00k} by repeating the proof of  \cite[Theorem~3.7]{CDK}.
More precisely, we consider the following three possible cases.
\begin{enumerate}[i)]
\item
Case {$\abs{x-y} \le \sqrt{t}<\diam \Omega$}:
In this case, we have
\[
r_0=\sqrt{t}=\abs{X-\hat{Y}}_{\sP}\quad\text{and}\quad\abs{x-y}^2/t \le 1.
\]
If $\abs{X'-X}_{\sP} < r_0/8$, then we take $r=r_0/4$ in \eqref{eq17.24} and use \eqref{eq5.18b} to get
\[
\abs{\vec K(x',y,t')-\vec K(x,y,t)}\le  C \abs{X'-X}_{\sP}^{\mu_1}\, r^{-n-\mu_1},
\]
which  implies \eqref{eq22.00k}.
If $r_0/8 \le \abs{X'-X}_{\sP} \le r_0/2$, then we have
\[
\abs{x'-y}\le 3r_0/2\quad\text{and}\quad r_0/2 \le \sqrt{t'} \le r_0 <\diam \Omega
\]
and thus, by \eqref{eq3.07ktx} we get
\[
\abs{\vec K(x',y,t')-\vec K(x,y,t)} \le \abs{\vec K(x',y,t')}+\abs{\vec K(x,y,t)} \le C r_0^{-n},
\]
which also implies \eqref{eq22.00k}.

\item
Case {$\sqrt{t}< \abs{x-y}$}:
In this case, $r_0=\abs{x-y} < \diam \Omega$.
Similar to \cite[Eq.~(5.22)]{CDK2}, for all $(z,s) \in Q_{-}(X,r_0/2)\cap \cQ$, we have
\begin{equation}			\label{eq5.26rv}
\abs{\vec K(z,y,s)} \le C t^{-n/2} \exp\left\{-\vartheta \abs{x-y}^2/4t\right\}.
\end{equation}
If $\abs{X'-X}_{\sP} < r_0/4$, then we take $r=r_0/2$ in \eqref{eq17.24} and use \eqref{eq5.26rv} to get
\[
\abs{\vec K(x',y,t')-\vec K(x,y,t)} \le C \abs{X'-X}_{\sP}^{\mu_1} r^{-\mu_1} t^{-n/2} \exp \left\{-\vartheta \abs{x-y}^2/4t \right\}.
\]
If $r_0/4 \le \abs{X'-X}_{\sP} \le r_0/2$, then use \eqref{eq3.07ktx} and \eqref{eq5.26rv} to get
\begin{equation*}
\abs{\vec K(x',y,t')-\vec K(x,y,t)}
\le C t^{-n/2} \exp\left\{-\vartheta \abs{x-y}^2/4t\right\}.
\end{equation*}
Therefore, we also obtain \eqref{eq22.00k} in this case.

\item
Case {$\diam \Omega \le \sqrt{t}$}:
In this case $r_0=d := \diam \Omega$, and the desired estimate \eqref{eq22.00k} becomes
\begin{equation}			\label{eq22.15}
\abs{\vec K(x',y,t')-\vec K(x,y,t)} \le C \abs{X'-X}_{\sP}^{\mu_1} \,d^{-n-\mu_1} \exp\left\{-\vartheta\abs{x-y}^2/4t\right\}.
\end{equation}
Since $t\ge d^2$, for all $(z,s)\in Q_{-}(X,r_0/2)\cap \cQ$, we have
\begin{equation}		\label{eq18.21}
\exp\left\{-\vartheta\abs{z-y}^2/s\right\}
\le e^{\vartheta/4} \exp\left\{-\vartheta\abs{x-y}^2/2t\right\}.
\end{equation}
If $\abs{X'-X}_{\sP} < r_0/4$, then we take $r=r_0/2$ in \eqref{eq17.24} and  use \eqref{eq18.21} to obtain \eqref{eq22.15}.
If $r_0/4 \le \abs{X'-X}_{\sP} \le r_0/2$, then by \eqref{eq3.07ktx} and \eqref{eq18.21}
\[
\abs{\vec K(x',y,t')-\vec K(x,y,t)}
\le C d^{-n} \exp\left\{-\vartheta \abs{x-y}^2/2t\right\}.
\]
Therefore, we also obtain \eqref{eq22.00k} in this case.
\end{enumerate}

The theorem is proved.
\hfill\qedsymbol

\subsection{Proof of Theorem~\ref{thm4}}

Assuming H1 and H2, we construct the Green's function $\vec G(x,y)$ for \eqref{MP} as follows.
Note that \eqref{eq2.03a} implies that there is a constant $\varrho$ such that for any $\vec u \in \vec V$, we have
\begin{equation}				\label{eq002v}
\norm{\vec u}_{L^2(\Omega)}\le \varrho \norm{\strain(\vec u)}_{L^2(\Omega)}.
\end{equation}
By utilizing \eqref{eq002v} and following the proof of \cite[Lemma~3.2]{DK09}, we get that for $x,y\in \Omega$ with $x\neq y$, we have
\[
\int_0^\infty \abs{\tilde{\vec K}(x,y,t)}\,dt<\infty,
\]
where $\tilde{\vec K}(\cdot,y,\cdot)$ is as in the proof of Theorem~\ref{thm1}.
We then define
\begin{equation}				\label{eq12.02b}
\vec G(x,y):= \int_0^\infty \tilde{\vec K}(x,y,t)\,dt.
\end{equation}
Then the symmetry relation \eqref{eq:E23} is an immediate consequence of \eqref{eq13.26e} once we show that $\vec G(x,y)$ is the Green's function.
We shall prove below that $\vec G(x,y)$ indeed enjoys the properties stated in Section~\ref{sec4.2}.
Denote
\[
\hat{\vec K}(x,y,t):= \int_0^t \tilde{\vec K}(x,y,s)\,ds
\]
so that we have
\[
\vec G(x,y)=\lim_{t\to \infty} \hat{\vec K}(x,y,t).
\]

\begin{lemma}			\label{lem4.10a}
The following holds uniformly for all $t>0$ and $y\in\Omega$.
\begin{enumerate}[i)]
\item
$\norm{\hat{\vec K}(\cdot,y,t)}_{L^p(B(y,d_y))} \le C_p (d_y^2+\varrho^2) d_y^{n/p-n},\quad \forall p \in \bigl[1,\frac{n+2}{n}\bigr)$.

\item
$\norm{\hat{\vec K}(\cdot,y,t)}_{L^{2(n+2)/n}(\Omega\setminus B(y,r))}
\le C(r^2+\varrho^2) r^{-\frac{n(n+4)}{2(n+2)}},\quad \forall r \in (0, d_y]$.

\item
$\norm{D\hat{\vec K}(\cdot,y,t)}_{L^p(B(y,d_y))}\le C_p (d_y^2+\varrho^2) d_y^{-1-n+n/p},\quad \forall p \in  \bigl[1,\frac{n+2}{n+1}\bigr)$.

\item
$\norm{D\hat{\vec K}(\cdot,y,t)}_{L^2(\Omega\setminus B(y,r))} \le C (r^2+\varrho^2) r^{-1-n/2},\quad \forall r \in (0, d_y]$.
\end{enumerate}
\end{lemma}
\begin{proof}
See \cite[Lemma~3.23]{DK09}.
\end{proof}

By the above lemma, elements $G_{ij}(x,y)$ of $\vec G(x,y)$ satisfy
\[
G_{ij}(\cdot,y) \in W^{1,1}(\Omega)\quad\text{and}\quad G_{ij}(\cdot, y) \in W^{1,2}(\Omega\setminus B(y,r))\;\text{ for any $r>0$}.
\]
Recall that columns of $\tilde{\vec K}(\cdot,y,t)$ are members of $\vec V$; see Lemma~\ref{lem5.18ah}.
Therefore, in the case when $D=\emptyset$, for any $\vec v \in \cR$, we have
\[
\int_\Omega \vec v(x)^T \tilde{\vec K}(x,y,t)\,dx =0
\]
and thus, we also have
\[
\int_\Omega \vec v(x)^T \vec G(x,y)\,dx =0.
\]
We have shown that $\vec G(x,y)$ satisfies the property i) in Section~\ref{sec4.2}.
For the proof of the property ii) in Section~\ref{sec4.2}, we refer to \cite[Section~3.2]{DK09}.
Finally, we show that the property iii) in Section~\ref{sec4.2} also holds.
Let $\vec f \in C^\infty(\bar\Omega)^n \cap \vec V$ and $\vec u$ be defined by the formula \eqref{choej}.
The integral \eqref{choej} is absolutely convergent by the property i)
of section \ref{sec4.2}.
Similarly, Lemma~\ref{lem4.10a} implies that
\[
\vec v(x,t) :=  \int_\Omega \hat{\vec K}(x,y,t) \vec f(y)\,dy
\]
is well defined.
Observe that
\begin{equation}		\label{eq:E29x}
\vec v(x,t)= \int_0^t \!\!\!\int_\Omega \tilde{\vec K}(x,y,s) \vec f(y)\,dy\,ds = \int_0^t \!\!\!\int_\Omega \tilde{\vec K}(x,y,t-s) \vec f(y) \,dy \,ds.
\end{equation}
Therefore, we have
\begin{align}
\label{eq:E29a}
\lim_{t\to\infty} \vec v(x,t)&=\int_\Omega \vec G(x,y) \vec f(y)\,dy=\vec u(x)\\
\label{eq:E29i}
\vec v_t(x,t)&=\int_\Omega \tilde{\vec K}(x,y,t) \vec f(y)\,dy.
\end{align}
By \eqref{ntilde}, the assumption that $\vec f \in \vec V$, and \eqref{eq13.37r}, we find from \eqref{eq:E29x} that $\vec v$ is, for any $T>0$, the weak solution in $\sV^{1,0}_2(\Omega \times (0,T))^n$ of the problem
\[
\left\{
\begin{array}{ll}
\vec v_t - L \vec v = \vec f \quad & \text{in }\; \Omega \times (0,T)\\
\vec v=0 & \text{on }\;  D\times (0,T)\\
\traction(\vec v) =0  & \text{on }\; N \times (0,T)\\
\vec v(\cdot, 0) = 0  & \text{on }\; \Omega.
\end{array} \right.
\]
By a similar reasoning, \eqref{eq:E29i} and the representation formula \eqref{eq10.06g} implies that $\vec v_t$ is, for any $T>0$, the weak solution in $\sV^{1,0}_2(\Omega\times(0,T))^n$ of the problem \eqref{eq10.08a} with $\vec \psi=\vec f$.
Then, we have (see \cite[Eq.~(3.42)]{DK09})
\begin{equation}		\label{eq:E29b}
\norm{\vec v_t(\cdot,t)}_{L^2(\Omega)} \le C e^{-\kappa_1\varrho^{-2}t}\norm{\vec f}_{L^2(\Omega)}, \quad \forall t>0.
\end{equation}
Observe that by \eqref{eq:E29x}, \eqref{eq:E29i}, and the assumption that $\vec f \in \vec V$, we have
\[
\vec v(\cdot, t) \in \vec V\quad\text{and}\quad \vec v_t(\cdot, t) \in \vec V\quad\text{for a.e. $t>0$}.
\]
Also, note that for any $\vec \phi=(\phi^1,\ldots,\phi^n)^T \in \vec V$ and for a.e. $t>0$, we have
\begin{equation}		\label{eq:E35}
\int_\Omega a^{\alpha\beta}_{ij} \frac{\partial v^j}{\partial x_\beta}(\cdot, t) \frac{\partial \phi^i}{\partial x_\alpha}\,dx= \int_\Omega f^i\phi^i \,dx -\int_\Omega v^i_t(\cdot\,t)\phi^i\,dx.
\end{equation}
Then, by setting $\vec \phi=\vec v(\cdot,t)$ in \eqref{eq:E35} and using \eqref{eq002v}, for a.e. $t>0$, we have
\begin{multline*}
\norm{\strain(\vec v(\cdot,t))}_{L^2(\Omega)}^2 \le C \left(\norm{\vec f}_{L^2(\Omega)}+\norm{\vec v_t(\cdot,t)}_{L^2(\Omega)}\right) \norm{\vec v(\cdot,t)}_{L^2(\Omega)} \\
\le C \norm{\vec f}_{L^2(\Omega)} \norm{\strain(\vec v(\cdot,t))}_{L^2(\Omega)},
\end{multline*}
where we have used \eqref{eq:E29b}.
Therefore, by \eqref{eq2.03a}, for a.e. $t>0$, we have
\[
\norm{\vec v(\cdot,t )}_{W^{1,2}(\Omega)} \le C \norm{\vec f}_{L^2(\Omega)}.
\]
Then, by the weak compactness of the space $W^{1,2}(\Omega)^n$ together with the fact that $\vec V$ is weakly closed in $W^{1,2}(\Omega)^n$, we find that there is a sequence $\set{t_m}_{m=1}^\infty$ tending to infinity and $\tilde{\vec u} \in \vec V$ such that
\[
\vec v(\cdot, t_m) \rightharpoonup \tilde{\vec u}\quad\text{weakly in }\;W^{1,2}(\Omega)^n.
\]
By \eqref{eq:E29a}, we must have $\vec u=\tilde{\vec u}\in \vec V$ and thus, for all $\vec \phi \in \vec V$, we get
\begin{equation}
\label{eq:E36}
\lim_{m\to\infty}\int_\Omega a^{\alpha\beta}_{ij} \frac{\partial v^j}{\partial x_\beta}(\cdot, t_m) \frac{\partial \phi^i}{\partial x_\alpha}\,dx = \int_\Omega a^{\alpha\beta}_{ij} \frac{\partial u^j}{\partial x_\beta} \frac{\partial \phi^i}{\partial x_\alpha}\,dx.
\end{equation}
Then, by \eqref{eq:E36}, \eqref{eq:E29b}, and \eqref{eq:E35}, for any $\vec \phi \in \vec V$, we obtain
\[
\int_\Omega a^{\alpha\beta}_{ij} \frac{\partial u^j}{\partial x_\beta} \frac{\partial \phi^i}{\partial x_\alpha}\,dx= \int_\Omega f^i \phi^i\,dx,
\]
which shows $\vec u$ is a weak solution in $\vec V$ of the problem \eqref{hee}; see the remark that appears above Theorem~\ref{thm4}.
Therefore, we verified that $\vec G(x,y)$ defined by the formula \eqref{eq12.02b} also satisfies the property iii) in Section~\ref{sec4.2}, and thus it is indeed the Green's function for \eqref{MP}.

Next, we assume H3 instead of H2 and proceed to prove the second part of the theorem.
In the rest of the proof, we shall denote
\[
{d}:= \diam\Omega.
\]
By Theorem~\ref{thm3}, we have the Gaussian bound \eqref{eq3.07ktx}.
In particular,  for $X=(x,t)\in \cQ$ satisfying $\sqrt{t}\le \diam \Omega$, we have
\begin{equation}					\label{eq618z}
\abs{\tilde{\vec K}(x,y,t)} \le C\abs{X-\hat{Y}}_{\sP}^{-n}.
\end{equation}
Similar to \cite[Eq.~(6.17)]{CDK2}, we have
\begin{equation}							\label{eq12.27g}
\abs{\tilde{\vec K}(x,y,t)} \le C r^{-n}e^{-\kappa_1\varrho^{-2} (t-2r^2)},\quad 
 t \ge 2r^2,\quad 0<r\le d.
 \end{equation}
We set $r := \frac{1}{2}\min(\varrho, d)$.
If $0<\abs{x-y} \le  r$, then by \eqref{eq12.02b}, we have
\begin{equation}		\label{eq6.21ee}
\abs{\vec G(x,y)} \le \int_0^{\abs{x-y}^2} + \int_{\abs{x-y}^2}^{2r^2}+\int_{2 r^2}^\infty \abs{\tilde{\vec K}(x,y,t)}\,dt =: I_1+I_2+I_3.
\end{equation}
It then follows from \eqref{eq618z} and \eqref{eq12.27g} that
\begin{align*}
I_1 &\le C\int_0^{\abs{x-y}^2}\abs{x-y}^{-n}\,dt \le C\abs{x-y}^{2-n},\\
I_2&\le C\int_{\abs{x-y}^2}^{2r^2}t^{-n/2}\,dt\le
\begin{cases}  C+C\ln (r/\abs{x-y}) & \text{if $n=2$,}
\\
C\abs{x-y}^{2-n} &\text{if $n\ge 3$.}
\end{cases} \\
I_3&\le C\int_{2r^2}^\infty r^{-n} e^{-\kappa_1\varrho^{-2} (t-2r^2)}\,dt\le C \varrho^2 r^{-n}.
\end{align*}
Combining all together we get that if $0<\abs{x-y}\le r$, then
\begin{equation}			\label{eq15.08b}
\abs{\vec G(x,y)} \le
\begin{cases}  C\left(1+(\varrho/r)^2+ \ln (r/d)+ \ln (d/\abs{x-y}) \right) & \text{if $n=2$,}\\
C\left(1+(\varrho/r)^2\right)\abs{x-y}^{2-n} &\text{if $n\ge 3$.}
\end{cases}
\end{equation}
In the case when $\abs{x-y} \ge r$, we estimate by \eqref{eq618z} and \eqref{eq12.27g} that
\begin{multline}		\label{eq6.23ww}
\abs{\vec G(x,y)} \le \int_0^{2r^2}\abs{\tilde{\vec K}(x,y,t)}\,dt+ \int_{2r^2}^\infty \abs{\tilde{\vec K}(x,y,t)}\,dt\\
\le C \int_0^{2r^2} \!\!\! r^{-n}\,dt + C \int_{2r^2}^\infty \!\!\! \ r^{-n} e^{-\kappa_1\varrho^{-2}(t-2r^2)}\,dt  \le Cr^{2-n}+ C \varrho^2 r^{-n}.
\end{multline}
By \eqref{eq15.08b} and \eqref{eq6.23ww}, we get \eqref{eq13.03k} and \eqref{eq13.03l}.
Finally, we turn to the proof of the estimate \eqref{eq15.13g}.
Because we assume H3, the conclusions of Theorem~\ref{thm3} are valid.
By \eqref{eq22.00k} and the definition \eqref{ntilde}, if $\abs{X-\hat{Y}}_{\sP} \le d$, then we have
\begin{multline}							\label{eq618ak}
\abs{\tilde{\vec K}(x',y,t)-\tilde{\vec K}(x,y,t)}\le \abs{\vec K(x',y,t)-\vec K(x,y,t)} + C\abs{x-x'} \\
\le C {\abs{x'-x}}^{\mu_1} \abs{X-\hat{Y}}_{\sP}^{-n-\mu_1}\quad\text{whenever }\;\abs{x-x'}<\tfrac{1}{2}\abs{x-y}.
\end{multline}
We claim that for $0<r \le d$ and $t>3r^2$, we have
\begin{equation}							\label{eq12.39p}
\abs{\tilde{\vec K}(x',y,t)-\tilde{\vec K}(x,y,t)} \le C \abs{x'-x}^{\mu_1}  r^{-n-\mu_1} e^{-\kappa_1\varrho^{-2} (t-2r^2)}
\end{equation}
whenever $\abs{x-x'}<\tfrac{1}{2}\abs{x-y}$.
Assume the claim \eqref{eq12.39p} for the moment.
Similar to \eqref{eq6.21ee}, in the case when $0<\abs{x-y} \le r:=  \frac{1}{2}\min(\varrho,{d})$, we get
\begin{multline*}
\abs{\vec G(x',y)-\vec G(x,y)} \le \\
\int_0^{\abs{x-y}^2} + \int_{\abs{x-y}^2}^{3r^2}+\int_{3 r^2}^\infty \abs{\tilde{\vec K}(x',y,t)-\tilde{\vec K}(x,y,t)}\,dt =:   I_1+I_2+I_3.
\end{multline*}
It follows from \eqref{eq618ak} that
\begin{align*}
I_1  &\le  C \abs{x'-x}^{\mu_1} \int_0^{\abs{x-y}^2} \abs{x-y}^{-n-\mu_1}\,dt \le C \abs{x'-x}^{\mu_1} {\abs{x-y}}^{2-n-\mu_1},\\
I_2  & \le C \abs{x'-x}^{\mu_1} \int_{\abs{x-y}^2}^\infty t^{-n/2-\mu_1/2}\,dt \le C \abs{x'-x}^{\mu_1} \abs{x-y}^{2-n-\mu_1}.
\end{align*}
Also, by \eqref{eq12.39p}, we obtain
\begin{multline*}
I_3 \le C r^{-n-\mu_1} \abs{x'-x}^{\mu_1} \int_{2r^2}^\infty  e^{-\kappa_1\varrho^{-2} (t-2r^2)}\,dt \\
\le  C \varrho^2 r^{-n-\mu_1} \abs{x'-x}^{\mu_1}
\le C \left(\frac{\varrho}{r}\right)^2 \abs{x'-x}^{\mu_1} \abs{x-y}^{2-n-\mu_1}.
\end{multline*}
Combining the above estimates together, we obtain \eqref{eq15.13g} when $0<\abs{x-y} \le r$.
In the case when $\abs{x-y} \ge r$, by using \eqref{eq618ak} and \eqref{eq12.39p}, we estimate
\begin{align*}
\abs{\vec G(x',y)-\vec G(x,y)} & \le \int_0^{4r^2}+\int_{4r^2}^\infty \abs{\tilde{\vec K}(x',y,t)-\tilde{\vec K}(x,y,t)}\,dt \\
& \le C \abs{x'-x}^{\mu_1} r^{2-n-\mu_1} +C \varrho^2 r^{-n-\mu_1} \abs{x'-x}^{\mu_1} \\
& \le  C(1+(\varrho/r)^2) \abs{x'-x}^{\mu_1} (r/{d})^{2-n-\mu_1} {d}^{2-n-\mu_1} \\
& \le  C(1+(\varrho/r)^2)  (r/{d})^{2-n-\mu_1}  \abs{x'-x}^{\mu_1} \abs{x-y}^{2-n-\mu_1}.
\end{align*}
Therefore, we also obtain \eqref{eq15.13g} when $\abs{x-y} \ge r$.
It only remains for us to prove the claim \eqref{eq12.39p}. The strategy is similar to the proof of \eqref{eq22.00k}.
Note that each column of $\tilde{\vec K}(\cdot,y,t)$ is a weak solution in $\sV_2(Q)$ of $\vec u_t - L \vec u=0$ provided that $Q\Subset \cQ\setminus \set{\hat{Y}}$.
Therefore, similar to \eqref{eq17.24}, for $0<r\le d$ and $t>3r^2$, we have
\[
\abs{\tilde{\vec K}(x',y,t)-\tilde{\vec K}(x,y,t)} \le C \abs{x'-x}^{\mu_1} r^{-n/2-1-\mu_1} \norm{\tilde{\vec K}(\cdot,y,\cdot)}_{\sL_2(Q_{-}(X,r)\cap \cQ)}
\]
whenever $\abs{x-x'} < r/2$.
Then by \eqref{eq12.27g}, we obtain \eqref{eq12.39p}.
\hfill\qedsymbol

\section{Appendix}

\subsection{Proof of Lemma~\ref{lem:lhp}}
\label{proof:lem:lhp}
We first show H3 implies H2.
Suppose $\vec u$ is a weak solution of $L \vec u=0$ in $B(x_0,r)\subset \Omega$.
By a well-known theorem of  Morrey \cite[Theorem~3.5.2]{Morrey}, we have
\[
[\vec u]_{\mu_0; B(x_0,r/2)}^2 \le C r^{2-n-2\mu_0}\norm{D \vec u}_{L^2(B(x_0,3r/4))}^2.
\]
By using the second Korn inequality, we get  Caccioppoli's inequality for $\vec u$, that is, for any $\vec \lambda \in \bR^n$ and $0<\rho \le  r/2$, we have
\[
\int_{B(x_0,\rho)} \abs{D\vec u}^2 \,dx \le C\rho^{-2} \left(1+(\diam \Omega)^2\right) \int_{B(x_0,2\rho)} \abs{\vec u-\vec \lambda}^2\,dx.
\]
Then, we get the estimate \eqref{IH} from \eqref{dirichlet}.

Next, we prove the estimate \eqref{eq2.11sc}.
We first establish a global version of \cite[Lemma~4.3]{Kim}.
For $\lambda \ge 0$, we denote by $L^{2,\lambda}(U)$ the linear space of functions $u \in L^2(U)$ such that
\[
\norm{u}_{L^{2,\lambda}(U)}=\left\{\sup_{\substack{x\in U\\0<r<\diam U}} r^{-\lambda} \int_{B(x,r)\cap U} \abs{u}^2\,dx\right\}^{1/2}<\infty.
\]
\begin{lemma}		\label{lem6.13apx}
Let $\vec f \in L^{2,\lambda}(B\cap\Omega)^n$, where $\lambda \ge 0$ and $B=B(x_0,R)$ with $x_0\in \bar\Omega$ and $0<R < \diam \Omega$.
Suppose $\vec u$ is a weak solution of
\[
\left\{
\begin{array}{lll}
L \vec u = \vec f  &\text{in}& B \cap \Omega, \\
\vec u=0 & \text{on}& B \cap D,\\
\traction(\vec u) =0  & \text{on} & B \cap N.
\end{array} \right. 
\]
If we assume H3, then, for $0 \le \gamma<\gamma_0 := \min(\lambda+4, n+2\mu_0)$, we have
\begin{equation}		\label{eq16.15ss}
r^{2-\gamma}\int_{B(x,r)\cap \Omega} \abs{D \vec u}^2\,dx \le C \left( \int_{B\cap \Omega} \abs{D \vec u}^2\,dx + \norm{\vec f}_{L^{2,\lambda}(B\cap \Omega)}^2 \right)
\end{equation}
uniformly for all $x \in \frac{1}{2}B \cap \bar \Omega$ and $0<r \le R/2$.
Here, $\frac{1}{2}B=B(x_0, R/2)$ and $C$ is a constant depending only on $n, \kappa_1, \kappa_2, \mu_0, A_0, \lambda, \gamma$, and $\Omega$.
We may take $\gamma=\gamma_0$ in \eqref{eq16.15ss} if $\gamma_0 <n$.
Moreover, if $\gamma<n$, then $\vec u \in L^{2,\gamma}(\frac{1}{2}B\cap \Omega)$ and
\begin{equation}			\label{eq16.16kk}
\norm{\vec u}_{L^{2,\gamma}(\frac{1}{2}B\cap \Omega)} \le C \left( \norm{\vec u}_{L^2(B\cap \Omega)} + \norm{D\vec u}_{L^2(B\cap \Omega)} + \norm{\vec f}_{L^{2,\lambda}(B\cap \Omega)}\right).
\end{equation}
\end{lemma}
\begin{proof}
Let $x\in B(x_0,R/2) \cap \partial\Omega$ and $0<r\le r_0 \wedge(R/2)$, where $r_0>0$ is such that
\begin{equation}	\label{app6.17}
\Omega\setminus \bigcup_{x\in \partial\Omega} B(x,r_0) \supset B(y,\delta_0)\quad\text{for some $y\in \Omega$ and $\delta_0>0$}.
\end{equation}
Denote
\[
\tilde N := N\cap B(x,r),\quad \tilde D := \partial(B(x,r)\cap \Omega) \setminus \tilde N
\]
and let $\vec v$ be a unique weak solution of the problem
\[
L\vec v = \vec f\;\text{ in }\; B(x,r)\cap \Omega,\quad \vec v =0\;\text{ on }\;\tilde D,\quad \traction(\vec v) =0\;\text{ on }\; \tilde N.
\]
By testing with $\vec v$  and using the Sobolev inequality, we get
\[
\int_{B(x,r)\cap \Omega} \abs{\strain(\vec v)}^2\,dy \le C \left( \int_{B(x,r)\cap \Omega} \abs{\vec f}^q\,dy\right)^{1/q} \left( \int_{B(x,r)\cap \Omega} \abs{\vec v}^{p}\,dy\right)^{1/p},
\]
where $1/p+1/q=1$ and $q=2n/(n+2)$ if $n \ge 3$ and $q=2/(\alpha+1)$ if $n=2$, where $\alpha \in [\mu_0,1)$.
 We extend $\vec v$ to $\Omega$ by setting $\vec v=0$ in $\Omega\setminus B(x,r)$.
Note that $\vec v\in W^{1,2}(\Omega)^n$ and $D\vec v=0$ on $\Omega\setminus B(x,r)$.
By the Sobolev inequality and H\"{o}lder's inequality, we then obtain
\[
\norm{\strain(\vec v)}_{L^2(\Omega)}^2 \le C r^\alpha \norm{\vec f}_{L^2(B\cap \Omega)}\norm{\vec v}_{W^{1,2}(\Omega)}.
\]
By the assumption \eqref{app6.17}, we have
\[
c \norm{D \vec v}_{L^2(\Omega)} \le \norm{\strain(\vec v)}_{L^2(\Omega)}, \quad \norm{\vec v}_{W^{1,2}(\Omega)} \le C \norm{D \vec v}_{L^2(\Omega)},
\]
for some constants $c, C>0$ independent of $\vec v$, $x$, and $r$.
Therefore, we have
\[
\int_{B(x,r)\cap \Omega} \abs{D \vec v}^2\,dy \le C r^{\lambda+2\alpha} \norm{\vec f}_{L^{2,\lambda}(B\cap \Omega)}^2.
\]
Let $\vec w:= \vec u-\vec v$.
Then, $\vec w$ is a weak solution of $L \vec w =0$ in $B(x,r)\cap \Omega$.
For $0<\rho <r$, we have
\begin{align*}
\int_{B(x,\rho)\cap \Omega} \abs{D\vec u}^2\,dy &\le 2 \int_{B(x,\rho)\cap \Omega} \abs{D\vec v}^2\,dy +  2 \int_{B(x,\rho)\cap \Omega} \abs{D\vec w}^2\,dy\\
&\le C  r^{\lambda+2\alpha} \norm{\vec f}_{L^{2,\lambda}(B\cap \Omega)}^2+C\left(\frac{\rho}{r}\right)^{n-2+2\mu_0} \int_{B(x,r)\cap \Omega} \abs{D\vec u}^2\,dy.
\end{align*}
Thus, by \cite[Lemma~2.1, p. 86]{Gi83}, we obtain \eqref{eq16.15ss} for $x\in B(x_0,R/2)\cap \partial\Omega$ and $0<r\le r_0 \wedge(R/2)$.
The general case of the estimate \eqref{eq16.15ss} is obtained by combining this case and the interior case, which is already covered by \cite[Lemma~4.3]{Kim}.
The estimate \eqref{eq16.16kk} is an easy consequence of the estimate \eqref{eq16.15ss} and the Poincar\'{e}'s inequality; see \cite[Proposition~1.2, p. 68]{Gi83}.
\end{proof}
Let $\vec u$ be a weak solution in $\sV_2(Q\cap \cQ)^n$ of \eqref{eq2.10tt}, where $Q=Q_{-}(X_0,r)$ with $X_0=(x_0,t_0) \in\cQ$ and $0<r \le \sqrt{t_0}\wedge \diam \Omega$.
Then, we have (see \cite[Lemma~4.2]{Kim})
\begin{align*}
\esssup_{t_0-(r/2)^2\le t\le t_0} \int_{B(x_0,r/2) \cap \Omega} \abs{\vec u_t}^2(x,t)\,dx &\le C r^{-6} \iint_{Q\cap \cQ} \abs{\vec u}^2\,dX,\\
\esssup_{t_0-(r/2)^2\le t\le t_0} \int_{B(x_0,r/2) \cap \Omega} \abs{D_x \vec u}^2(x,t)\,dx &\le C r^{-4} \iint_{Q\cap \cQ} \abs{\vec u}^2\,dX.
\end{align*}
Also, we have (cf. \cite[Lemma~8.6]{CDK2})
\[
\int_{Q\cap \cQ} \abs{\vec u -\vec \lambda}^2\,dX \le Cr^2 \int_{Q\cap \cQ} \abs{D_x \vec u}^2\,dX;\quad \vec \lambda:= \fint_{Q\cap \cQ} \vec u\,dX.
\]
By using Lemma~\ref{lem6.13apx} and the above inequalities, we repeat the proof of \cite[Theorem~3.3]{Kim} with obvious modifications to get
\begin{equation}		\label{lex}
[\vec u]_{\mu_1,\mu_1/2; \frac{1}{2}Q \cap \cQ} \le A_1 r^{-\mu_1-(n+2)/2} \norm{\vec u}_{\sL_2(Q\cap \cQ)}.
\end{equation}
Finally, we show that the above estimate and the condition \eqref{typea}  implies
\begin{equation}		\label{lex2}
\abs{\vec u}_{0; \frac{1}{2}Q \cap \cQ} \le A_1 r^{-(n+2)/2} \norm{\vec u}_{\sL_2(Q\cap \cQ)}.
\end{equation}
For $Y\in \frac{1}{4}Q\cap \cQ$ and $Z\in \cQ$ satisfying $\abs{Z-Y}_{\sP}<r/4$, we have
\[
\abs{\vec u(Y)}^2 \le 2^{1-4\mu_1} r^{2\mu_1} [\vec u]_{\mu_1,\mu_1/2;,\frac{1}{2}Q\cap\cQ}^2+2\abs{\vec u(Z)}^2.
\]
By taking the average over $Z$ and using \eqref{lex}, we get
\[
\abs{\vec u(Y)}^2 \le \left(2^{1-4\mu_1}A_1^2 +\beta^{-1} 2^{2n+5}\right) r^{-n-2}\norm{\vec u}_{\sL_2(Q\cap \cQ)}^2.
\]
By adopting a covering argument, we obtain \eqref{lex2}, and thus \eqref{eq2.11sc}.
\hfill\qedsymbol

\subsection{Proof of Lemma~\ref{lem3a}}
\label{proof:lem3a}
Recall that $\tilde{\vec u}$ is the weak solution of the problem \eqref{eq5.06jj}, where $\vec f$ is understood as an element of $L^1((a,b);\vec V')$.
We define
\[
\tilde{\vec f}(x,t) := \vec f(x,t)- \pi_\cR[\vec f(\cdot,t)](x),\quad (x,t)\in Q:= \Omega\times (a,b).
\]
Note that $\vec f = \tilde{\vec f}$ in $L^1((a,b); \vec V')$.
Also, for all $t\in(a,b)$, we have
\[
\norm{\pi_\cR\vec f(\cdot, t)}_{L^\infty(\Omega)} \le C \norm{\pi_\cR\vec f(\cdot, t)}_{L^2(\Omega)} \le C \norm{\vec f(\cdot,t)}_{L^2(\Omega)}\le C \abs{\Omega}^{1/2} \norm{\vec f}_{\sL_\infty(Q)},
\]
where we used the fact that all norms in $\cR$ are equivalent and that $\pi_\cR$ is an orthogonal projection.
Therefore, we have
\begin{equation}			\label{eq16.00}
\norm{\tilde{\vec f}}_{\sL_\infty(Q)} \le C \norm{\vec f}_{\sL_\infty(Q)}.
\end{equation}
Let $\vec u$ be the weak solution in $\sV_2^{1,0}(Q)^n$ of the problem
\[
\left\{\begin{array}{ll}
-\vec u_t - L \vec u =\tilde{\vec f} \quad & \text{in }\; \Omega \times (a,b)\\
\vec u=0 & \text{on }\;  D\times (a,b)\\
\traction(\vec u) =0  & \text{on }\; N \times (a,b)\\
\vec u(\cdot, b) = 0 &  \text{on }\; \Omega.
\end{array}
\right. 
\]
Then it is easy to see that $\vec u(\cdot,t) \in \vec V$ for a.e. $t\in (a,b)$ and thus $\vec u$ is a weak solution of the problem \eqref{eq5.06jj}.
Therefore, by the uniqueness, we conclude that $\tilde{\vec u} = \vec u$.
In particular, $\tilde{\vec u}$ is a weak solution in $\sV_2(Q)^n$ of $-\vec u_t+ L\vec u= \tilde{\vec f}$.
Because of the inequality \eqref{eq16.00}, it is then enough to prove the second part of the lemma and we refer to \cite[Section~3.2]{CDK} for its proof.
\hfill\qedsymbol

\end{document}